\documentclass[11pt]{amsart}

\usepackage{tikz}
\usepackage[latin1]{inputenc}
 \usepackage[T1]{fontenc}

\usepackage{amssymb,amsmath}
\usepackage{multirow}
\usepackage{enumerate}

\newtheorem{theo}{Theorem}
\newtheorem{lem}{Lemma}
\newtheorem{propo}{Proposition}
\newtheorem{rem}{Remark}

\newenvironment{defi}{\noindent{\bf Definition.}}{\\}
\newenvironment{ask}{\noindent{\bf Question.}}{\\}

\def\longto{\longrightarrow}

\def\codim{{\rm codim}}
\def\Chi{X}

\def\Wt{{\rm Wt}}
\def\Ta{{\mathcal T}}
\def\hTa{{\hat\Ta}}
\def\hP{{\hat P}}\def\hw{{\hat w}}\def\hW{{\hat W}}\def\hL{{\hat
    L}}\def\hT{{\hat T}}\def\hl{{\hat l}}\def\hB{{\hat B}}
 \def\hs{{\hat s}}   
 \def\hvarpi{\hat{\varpi}} 
 \def\hvarepsilon{\hat{\varepsilon}}
    
\def\hnu{{\hat \nu}}
\def\hZ{{\hat Z}}
\def\lg{{\mathfrak g}}\def\lt{{\mathfrak t}}
\def\hlg{{\hat{\lg}}}
\def\so{{\mathfrak so}}
\def\halpha{{\hat\alpha}}

\def\RR{{\mathbb R}}\def\QQ{{\mathbb Q}}\def\ZZ{{\mathbb Z}}
\def\CC{{\mathbb C}}

\def\SL{{\rm SL}}
\def\GL{{\rm GL}}
\def\Sp{{\rm Sp}}\def\Spin{{\rm Spin}}
\def\SO{{\rm SO}}
\def\Gr{{\rm Gr}}\def\hG{{\hat G}} 
\def\Hom{{\rm Hom}}

\makeatletter
\def\revddots{\mathinner{\mkern1mu\raise\p@\vbox{\kern7\p@\hbox{.}}\mkern2mu\raise4\p@\hbox{.}\mkern2mu\raise7\p@\hbox{.}\mkern1mu}}
\makeatother

\def\base{{\mathcal B}}
\def\diag{{\rm diag}}
\def\Fl{{\mathcal Fl}}

\def\videGB{
    \draw (0,0) -- (1,0) -- (1,1) -- (0,1) -- cycle;
\draw (3,0) -- (4,0) -- (4,1) -- (3,1) -- cycle;
\draw (1.5,0) -- (2.5,0) -- (2.5,-1) -- (1.5,-1) -- cycle;
}

\def\grisun{\draw [fill=gray!60]  (0,0) -- (1,0) -- (1,1) -- (0,1) -- cycle;}
\def\grisdeux{\draw [fill=gray!60] (3,0) -- (4,0) -- (4,1) -- (3,1) -- cycle;}
\def\gristrois{\draw [fill=gray!60] (1.5,0) -- (2.5,0) -- (2.5,-1) --
  (1.5,-1) -- cycle;}

\def\iempty{
    \draw (0,0) -- (2,0) -- (2,3) -- (0,3) -- cycle;
\draw (0.5,0) -- (1.5,0) -- (1.5,-1) --(0.5,-1) -- cycle;
\draw (0,1) -- (2,1) (0,2) -- (2,2) (0,3) -- (2,3) (1,0) -- (1,3);
}

\def\iemptygray{
    \draw (0,0) -- (2,0) -- (2,3) -- (0,3) -- cycle;
\draw (0,1) -- (2,1) (0,2) -- (2,2) (0,3) -- (2,3) (1,0) -- (1,3);
\draw [fill=gray!60] (0.5,0) -- (1.5,0) -- (1.5,-1) --(0.5,-1) -- cycle;
}

\def\iundeux{
  \begin{tikzpicture}[scale=0.2]
    \iempty
  \end{tikzpicture}
}

\def\iuntrois{
  \begin{tikzpicture}[scale=0.2]
\draw [fill=gray!60] (1,2) -- (1,3) -- (2,3) -- (2,2) -- cycle;
    \iempty
  \end{tikzpicture}
}

\def\iuncinq{
  \begin{tikzpicture}[scale=0.2]
\draw [fill=gray!60] (1,1) -- (1,3) -- (2,3) -- (2,1) -- cycle;
    \iempty
  \end{tikzpicture}
}

\def\iunsix{
  \begin{tikzpicture}[scale=0.2]
\draw [fill=gray!60] (1,0) -- (1,3) -- (2,3) -- (2,0) -- cycle;
    \iempty
  \end{tikzpicture}
}

\def\ideuxtrois{
  \begin{tikzpicture}[scale=0.2]
\draw [fill=gray!60] (0,2) -- (0,3) -- (2,3) -- (2,2) -- cycle;
    \iempty
  \end{tikzpicture}
}

\def\ideuxcinq{
  \begin{tikzpicture}[scale=0.2]
\draw [fill=gray!60] (0,2) -- (0,3) -- (2,3) -- (2,1) -- (1,1) -- (1,2) -- cycle;
    \iempty
  \end{tikzpicture}
}

\def\ideuxsept{
  \begin{tikzpicture}[scale=0.2]
\draw [fill=gray!60] (1,0) -- (1,3) -- (2,3) -- (2,0) -- cycle;
    \iemptygray
  \end{tikzpicture}
}

\def\itroissix{
   \begin{tikzpicture}[scale=0.2]
\draw [fill=gray!60] (0,2) -- (0,3) -- (2,3) -- (2,1) -- (1,1) -- (1,2) -- cycle;
    \iemptygray
  \end{tikzpicture}
}

\def\itroissept{
   \begin{tikzpicture}[scale=0.2]
\draw [fill=gray!60] (0,2) -- (0,3) -- (2,3) -- (2,0) -- (1,0)-- (1,2) -- cycle;

    \iemptygray
  \end{tikzpicture}
}

\def\icinqsix{
  \begin{tikzpicture}[scale=0.2]
\draw [fill=gray!60] (0,1) -- (0,3) -- (2,3) -- (2,1) -- cycle;
    \iemptygray
  \end{tikzpicture}
}

\def\icinqsept{
  \begin{tikzpicture}[scale=0.2]
\draw [fill=gray!60] (1,0) -- (1,1) -- (0,1) -- (0,3) -- (2,3) -- (2,0) -- (1,0) -- (1,1) -- cycle;
    \iemptygray
  \end{tikzpicture}
}

\def\isixsept{
  \begin{tikzpicture}[scale=0.2]
    \draw [fill=gray!60] (0,0) -- (2,0) -- (2,3) -- (0,3) -- cycle;
\iemptygray
  \end{tikzpicture}
}

\begin{document}

\title{The saturation property for branching rules -- Examples}
\author{B. Pasquier,
N. Ressayre
}

\begin{abstract}
For a few pairs $(G\subset \hG)$ of reductive groups, we study the decomposition of irreducible $\hG$-modules into  $G$-modules. In particular, we observe the saturation property for all of these pairs. 
\end{abstract}

\maketitle

\section{Introduction}\label{sec:intro}

Let $G$ be a complex connected reductive group. Studying the tensor product decomposition of  irreducible representations
of $G$ is a very classical and important problem in representation theory.
More recently, 
Klyachko's contribution \cite{Kly} of the Horn problem of characterizing the possible eigenvalues of three Hermitian matrices
whose sum is zero, motivated the  so-called saturation conjecture for the group $G=\GL_n$.
This conjecture was solved by Knutson and Tao \cite{KT:saturation} and studied for others groups \cite{DW:saturation,KM,BK:satoddorthsymp,Sam}.

The  tensor product of two irreducible representations of $G$ is an irreducible representation of $\hG=G\times G$. 
In particular, tensor product decomposition is a particular case of 
the following branching problem. Assume that $G$ is embedded in a 
bigger connected reductive group $\hat{G}$. 
Then we are interested in decomposing irreducible representations of 
$\hat G$ as a sum of irreducible $G$-modules.
The aim of this note is to state a saturation property in this more general setting and to study some explicit examples using some computer calculation with
\cite{4ti2} and \cite{sage}.

\subsection{Overview of saturation property for tensor product decomposition}

We fix a Borel subgroup $B$ and a maximal torus $T\subset B$ in $G$. If $\nu$ is a dominant weight, $V_G(\nu)$ denotes the irreducible 
representation of highest weight $\nu$.
For any $G$-module $V$, the set of fixed points is denoted by $V^G$.
The saturation property for $\GL_n$ can be stated as follows.

\begin{theo}[Knutson-Tao]\label{th:satGL}
  Let $\nu_1$, $\nu_2,$ and $\nu_3$ be three dominant weights of $G=\GL_n(\CC)$. 

If $(V_G(N\nu_1)\otimes V_G(N\nu_2)\otimes V_G(N\nu_3))^G\neq\{0\}$ for some positive integer $N$ then
$(V_G(\nu_1)\otimes V_G(\nu_2)\otimes V_G(\nu_3))^G\neq\{0\}$.
\end{theo}

The first proof \cite{KT:saturation} of Theorem~\ref{th:satGL} due to Knutson and Tao uses a combinatorial model for 
Littlewood-Richardson coefficients called honeycombs. 
Derksen and Weyman reproved \cite{DW:saturation} this result using representations of quivers and
 Kapovich and Millson obtained a proof \cite{KM} using the geometry of Bruhat-Tits buildings.

Assume now that $G$ is semisimple and let $\Lambda_R$ denote its
root lattice. Theorem~\ref{th:satGL} can be restated as follows.

\begin{theo}[Knutson-Tao]\label{th:satSL}
  Let $\nu_1$, $\nu_2,$ and $\nu_3$ be three dominant weights of $G=\SL_n(\CC)$. 

If $(V_G(N\nu_1)\otimes V_G(N\nu_2)\otimes V_G(N\nu_3))^G\neq\{0\}$ for some positive integer $N$ 
and $\nu_1+\nu_2+\nu_3\in\Lambda_R$, then
$(V_G(\nu_1)\otimes V_G(\nu_2)\otimes V_G(\nu_3))^G\neq\{0\}$.
\end{theo}

We say that the tensor product decomposition for $\SL_n$ satisfies the {\it saturation property}.
The best known uniform generalization of Theorem~\ref{th:satSL} to any simple group $G$ is 

\begin{theo}[Kapovich-Millson \cite{KM}]\label{th:satGss}
  Let $\nu_1$, $\nu_2,$ and $\nu_3$ be three dominant weights of the simple group $G$. 
Let $k$ be the least
common multiple of the coefficients of the highest root of $G$ written in terms of simple roots.

If $(V_G(N\nu_1)\otimes V_G(N\nu_2)\otimes V_G(N\nu_3))^G\neq\{0\}$ for some positive integer $N$ and $\nu_1+\nu_2+\nu_3\in\Lambda_R$, then
$(V_G(k^2\nu_1)\otimes V_G(k^2\nu_2)\otimes V_G(k^2\nu_3))^G\neq\{0\}$.

\end{theo}

Observe that for $G=\SL_n$, $k=1$. Belkale and Kumar \cite{BK:satoddorthsymp}  and Sam \cite{Sam} 
obtained better constants than $k^2$ for classical groups.



Two important conjectures in the topic are still open. 
The first one asserts that tensor product decompositions for simply-laced
groups satisfy the saturation property. The second one asserts that Theorem~\ref{th:satSL}
is satisfied for any $G$ if the weights are regular.  

\subsection{Saturation property for branching problem}

We fix maximal tori $T$ and $\hT$ and Borel subgroups $B$ and $\hB$ of $G$ and
$\hG$ such that $ \hB\supset \hT\supset T\subset B\subset\hB$.
We consider the set
$
LR(G,\hG)
$
of pairs $(\nu,\hnu)$ of dominant weights such that $(V_G(\nu)\otimes V_\hG(\hnu))^G\neq\{0\}$, that is, such that 
 $V_G(\nu)^*$ is a sub-$G$-module of $V_\hG(\hnu)$.
By definition $LR(G,\hG)$ is a subset of the character group $\Chi(T\times \hT)$ of $T\times\hT$.
By a result of Brion and Knop (see \cite{Elash}), $LR(G,\hG)$ is a {\it finitely generated subsemigroup}
of the lattice $\Chi(T\times \hT)$.
We say that the pair $(G,\hG)$ has the saturation property if $LR(G,\hG)$ is the intersection of some convex cone with some lattice. To make this more precise we consider
the subgroup $\ZZ LR(G,\hG)$ of  $\Chi(T\times \hT)$ generated by $LR(G,\hG)$.
The following statement describes the group $\ZZ LR(G,\hG)$.

\begin{theo}\label{th:LRgroup}
Let $\hZ$ denote the center of $\hG$.
Suppose that every connected, closed and normal subgroup of $\hG$ contained 
in $G$ is trivial.
  Then the group $\ZZ LR(G,\hG)$ is the set of pairs $(\nu,\hnu)\in \Chi(T)\times \Chi(\hT)$ such that 
$$
\nu(t).\hnu(t)=1
$$
for any $t\in \hZ\cap G$.
\end{theo}

Note that Theorem~\ref{th:LRgroup} is announced in \cite{Brion:Bourbaki}.

\begin{rem}
The hypothesis done in Theorem~\ref{th:LRgroup} is not very restrictive. Indeed, for any pair $(G,\hG)$, let $H$ be the maximal connected, closed and normal subgroup of $\hG$ contained in $G$. Then, by taking a finite cover of $\hG$ and the neutral component of the inverse image of this cover in $G$, we can suppose that $\hG=H\times \hG_0$ and $G=H\times G_0$. Then $LR(G,\hG)=LR(G_0,\hG_0)$ and $(G_0,\hG_0)$ satisfies the hypothesis of Theorem~\ref{th:LRgroup}.

\end{rem}

\begin{defi}
  The semigroup $LR(G,\hG)$ (or the pair $(G,\hG)$) is said to have the {\it saturation property}
if for any pair of dominant weights $(\nu,\hnu)$ such that
\begin{enumerate}
\item $\forall t\in \hZ\cap G,\quad \nu(t).\hnu(t)=1$ and
\item $\exists N>0,\quad (V_G(N\nu)\otimes V_\hG(N\hnu))^G\neq\{0\}$,
\end{enumerate}
we have
$$
(V_G(\nu)\otimes V_\hG(\hnu))^G\neq\{0\}.
$$
\end{defi}

\subsection{Examples}
Guessing that this work can help to understand better  the saturation property for
branching rules (and maybe even for the tensor product decomposition), we study
this property in detail for some examples. We make a particular attention to the 
case when $G$ is spherical of minimal rank in $\hG$ (see \cite{Spherangmin} for 
a classification). Our motivation is that these branching rules have common 
properties with the tensor product decomposition 
(see for example \cite{MPR,MPR2}). 
We surprisingly observed that all the computed examples have the saturation 
property. 

\begin{theo}\label{th:explesat}
The pairs $(\Spin_{2n-1},\Spin_{2n})$, $(\SL_3, G_2)$, $(G_2,\Spin_7)$, $(\Spin_9,F_4)$, $(F_4,E_6)$, 
$(\Sp_4,\SL_4)$, $(\Sp_6,\SL_6)$,$(\Sp_8,\SL_8)$,$(\Sp_{10},\SL_{10})$  have
the saturation property.
\end{theo}

Along the way, we compute many other datum attached to the semigroup $LR(G,\hG)$:
inequalities and rays for the generated cone, Hilbert basis.\\

Regarding Theorem~\ref{th:explesat}, it is natural to extend the conjecture of saturation of tensor product decompositions of simply laced groups. Indeed,
consider the set $\Wt_T(\hlg/\lg)$ of  non trivial weights of $T$ in the quotient $\hlg/\lg$
of the Lie algebras of $\hat G$ and $G$. 

\bigskip
\begin{ask}
Assume that $\hG/G$ is spherical of minimal rank and that $W$ acts transitively on $\Wt_T(\hlg/\lg)$.  
  
  Does  $(G,\hat G)$ have the saturation property?
\end{ask}

This paper reduces the above question to two cases: the tensor product decomposition for
simple simply laced groups (the classical conjecture) and $Sp_{2n}\subset Sl_{2n}$. 
This last case is checked for $n\leq 5$.

\section{Proof of Theorem~\ref{th:LRgroup} and a first example}\label{sec:LRgroup}


\begin{lem}\label{lem:LRgroup1}
Let $X$ be an algebraic variety and let $G$ be a reductive group acting
 on $X$ with a fixed point $x$. Then the actions of $G$ on $X$ and on $T_xX$ have the same kernel.
\end{lem} 

\begin{proof}
It is enough to prove that if an element $g$ of $G$ acts trivially on $T_xX$, then it also acts trivially on the local ring $\mathcal{O}_{X,x}$. Denote by $\mathfrak{m}_x$ the maximal ideal of $\mathcal{O}_{X,x}$. Then $g$ acts trivially on $\mathfrak{m}_x/\mathfrak{m}_x^2=(T_xX)^*$. 
It also acts trivially on each symmetric power $S^n(\mathfrak{m}_x/\mathfrak{m}_x^2)$ and each quotient $\mathfrak{m}_x^n/\mathfrak{m}_x^{n+1}$. Now, since $\mathcal{O}_{X,x}/\mathfrak{m}_x^{n+1}$ is a rational $G$-module of finite dimension, it is semi-simple and then $g$ acts trivially on it. We conclude by the fact that $\cap_{n\geq 1}\mathfrak{m}_x^n=\{0\}$.
\end{proof}

Let $U$ (resp. $\hat{U}$) be the unipotent radical of $B$ (resp. $\hB$) and let $\hat{U}^-$ be the unipotent radical of the Borel $\hB^-$ opposite to $\hB$.  And denote by $\mathfrak{g}$, $\hat{\mathfrak{g}}$, $\mathfrak{u}$, $\hat{\mathfrak{u}}$, $\mathfrak{t}$ and $\hat{\mathfrak{t}}$ the Lie algebras of $G$, $\hG$, $U$, $\hat{U}$, $T$ and $\hT$ respectively. 

If $V$ is a $G$-module, then since $T$ normalizes $U$, $T$ acts on $V^U$. We denote by $V^U_\nu$ the subspace of $V^U$ on which $T$ acts with weight $\nu$. We generalize in a natural way this notation to $G\times\hG$-modules $V$ with the unipotent radical $U\times \hat{U}^-$ of $B\times \hB^-$.


\begin{lem}\label{lem:LRgroup2}
Consider the actions by right multiplications of $U$ and $\hat U^-$ on $G$ and $\hG$.
The morphism of algebras given by:
$$\begin{array}{ccc}
(\CC[G]^U\otimes\CC[\hG]^{\hat{U}^-})^G & \longrightarrow & \CC[\hG]^{U\times \hat{U}^-}\\
\sum_i\phi_i\otimes\psi_i & \longmapsto & \sum_i\phi_i(e)\psi_i
\end{array}$$
where $G$ acts diagonally on $\CC[G]^U\otimes\CC[\hG]^{\hat{U}^-}$ and where $e$ is the unity in $G$, is an isomorphism.

In particular, $((V_G(\nu))^*\otimes V_\hG(\hat{\nu}))^G$ is isomorphic to $\CC[\hG]^{U\times \hat{U}^-}_{\nu,-\hat{\nu}}$.
\end{lem}

\begin{proof}
The inverse of the morphism comes from:

$$\begin{array}{ccc}
\CC[\hG] & \longrightarrow & \CC[G\times \hG]\simeq\CC[G]\otimes\CC[\hG]\\
f & \longmapsto & ((g,\hat{g})\mapsto f(g\hat{g})).
\end{array}$$

For the last statement, we use the decompositions of $\CC[G]$ and $\CC[\hG]$: 
$$\CC[G]=\bigoplus_{\nu\in X(T)^+}V_G(\nu)^*\otimes V_G(\nu)\,\mbox{ and }\,\CC[\hG]=\bigoplus_{\hnu\in X(\hat T)^+}V_\hG(\hnu)\otimes V_\hG(\hnu)^*.$$
Remark also that $V_G(\nu)^U$ is a line on which $T$ acts with weight $\nu$ and $(V_\hG(\hat{\nu})^*)^{\hat{U}^-}$ is a line on which $\hT$ acts with weight $-\hat{\nu}$.
\end{proof}

\begin{proof}[Proof of Theorem~\ref{th:LRgroup}]
Denote by $\nu^*$ the highest weight of $V_G(\nu)^*$.
Then we define 
$$
H:=\{(t,\hat{t})\in T\times\hT\,\mid\,\nu^*(t)=\hnu(\hat{t})\mbox{ for any }(\nu,\hnu)\in LR(G,\hG)\}.
$$ 
By Lemma~\ref{lem:LRgroup2}, $H$ is the kernel of the action of $T\times\hT$ on $\CC[\hG]^{U\times \hat{U}^-}$, 
hence also on $\CC(\hG)^{U\times \hat{U}^-}$.
The Bruhat decomposition gives an open immersion of $\hat{U}\times\hT\times\hat{U}^-$ in $\hG$. 
Then $\CC(\hG)^{U\times \hat{U}^-}$ is isomorphic to $\CC(\hat{U}/U\times\hT)$.
Then $H$ is the kernel of the action of $T\times\hT$ on $\hat{U}/U\times\hT$ given by $(t,\hat{t})\cdot (\hat{u}U,\hat{x})=(t\hat{u}t^{-1}U,t\hat{x}\hat{t}^{-1})$. We deduce easily that $H=\{(t,t)\in T\times T\,\mid\,t\in H'\}$, where $H'$ is the kernel of the action (by conjugation) of $T$ on $\hat{U}/U$.
Since $U/U$ is fixed by this action, by Lemma~\ref{lem:LRgroup1}, $H'$ is also the kernel of the action of $T$ on the quotient of Lie algebras $\hat{\mathfrak{u}}/\mathfrak{u}$ and then also the kernel of the action on $\hat{\mathfrak{g}}/\mathfrak{g}\simeq (\hat{\mathfrak{u}}/\mathfrak{u})\oplus(\hat{\mathfrak{t}}/\mathfrak{t})\oplus (\hat{\mathfrak{u}}/\mathfrak{u})^*$. Still with Lemma~\ref{lem:LRgroup1}, $H'$ is the kernel of the action (by conjugation) of $T$ on $\hG/G$, and we obtain 
$$
H'=T\cap\bigcap_{\hat{g}\in\hG}\hat{g}G\hat{g}^{-1}.
$$

Now, $\cap_{\hat{g}\in\hG}\hat{g}G\hat{g}^{-1}$ is a closed and normal subgroup of $\hG$ contained in $G$.  
Hence the hypothesis implies that the intersection  $\cap_{\hat{g}\in\hG}\hat{g}G\hat{g}^{-1}$ is finite (and normal). 
Then, since $\hG$ is reductive, it is contained in $\hZ$, and $H'\subset\hZ$. 
Conversely $\hZ$ acts trivially on $\hG/G$, so that $H'=\hZ\cap T=\hZ\cap G$.
Finally,
$$
H=\{(t,t)\in T\times T\,\mid\,t\in\hZ\cap G\}.
$$

We then deduce that the group $\ZZ LR(G,\hG)$ is the set of pairs $(\nu,\hnu)\in \Chi(T)\times \Chi(\hT)$ such that 
$$
\nu^*(t)=\hnu(t)
$$
for any $t\in \hZ\cap G$.

But $\nu^*=-w_0\nu$, where $w_0$ is the longest element of the Weyl group of $G$, and then for all element of the center of $G$ (in particular for all $t\in\hZ\cap G$), we have $\nu^*(t)=-w_0\nu(t)=-\nu(w_0tw_0^{-1})=-\nu(t)=\nu(t^{-1})$. This concludes the proof of Theorem~\ref{th:LRgroup}.
\end{proof}

{\bf Example:} Here $G=\Spin_{2n-1}$ and $\hG=\Spin_{2n}$.

We denote by $(\varepsilon_1,\dots,\varepsilon_n)$ the standard (orthogonal) basis of the weight lattice of the maximal torus  of $\SO_{2n}$ 
(with Bourbaki's notation).
Then $X(\hT)$ is the set of $\hat \nu=\hnu_1\varepsilon_1+\cdots+\hnu_n\varepsilon_n$ for some rational numbers $\hnu_i$ 
such that $(2\hnu_1,\dots,2\hnu_{n})$ are integers of same parity.
Similarly $X(T)$ is the set of $\nu=\nu_1\varepsilon_1+\cdots+\nu_{n-1}\varepsilon_{n-1}$
such that $(2\nu_1,\dots,2\nu_{n-1})$ are integers of same parity.
The weights $\nu$ and $\hnu$ are dominant if and only if 
$$
\nu_1\geq\nu_2\geq\cdots \geq\nu_{n-1}\geq 0\quad\mbox{and}\quad
\hnu_1\geq\hnu_2\geq\cdots \geq\hnu_{n-1}\geq
|\hnu_n|.
$$

The center of $G$ is isomorphic to $\ZZ/2\ZZ$.
By Theorem~\ref{th:LRgroup}, $(\nu,\hnu)$ belongs to $\ZZ LR(\Spin_{2n-1},\Spin_{2n})$ if and only if the integers $2\nu_i$ and $2\hnu_j$ have all the same parity.

The convex cone generated by $LR(\Spin_{2n-1},\Spin_{2n})$ in $(\Chi(T)\times \Chi(\hT))_\QQ$ is already given in~\cite{FH} by the following irredundant $2n-1$ inequalities:
$$
\hnu_1\geq\nu_1\geq\hnu_2\geq\nu_2\geq\cdots \geq\nu_{n-1}\geq
|\hnu_n|,
$$
in particular it is a simplex.
Then, an Hilbert basis of this cone in $\ZZ LR(\Spin_{2n-1},\Spin_{2n})$ is easily computable and given by all the following sequences with at least two 0 and one 1
$$
1\geq \cdots \geq 1\geq 0\geq\cdots \geq |0|,
$$
and the two sequences
$$
\frac 1 2\geq \cdots \geq \frac 1 2 \geq \frac 1 2\mbox{ and }\frac 1 2\geq \cdots \geq \frac 1 2 \geq -\frac 1 2.
$$
These $2n-1$ elements correspond to the following decompositions :

\begin{itemize}
\item $V(\hat\varpi_i)=V(\varpi_{i-1})\oplus V(\varpi_i)$ for $1\leq
  i\leq n-2$; (by convention $V(\varpi_0)$ is the trivial representation $\CC$)
\item $V(\hat\varpi_{n-1}+\hat\varpi_{n})$ contains $V(\varpi_{n-2})$;
\item $V(\hat\varpi_{n-1})=V(\varpi_{n})$;
\item $V(\hat\varpi_{n})=V(\varpi_{n})$.
\end{itemize}

We can conclude that the pair $(\Spin_{2n-1},\Spin_{2n})$ has the saturation property.
We also remark that, any inequality coming from dominance is redundant.\\

In all others examples we need another strategy to study the semigroup, the cone and the saturation property. We explain this in the following section.

\section{Method to study several examples}

\subsection{Levi-movability}

Recall that  $G\subset\hG$ are two complex connected reductive groups.
Let $\lambda$ be a one-parameter subgroup (1-ps) of $T$.
The set of $g\in G$ such that $\lim_{t\to 0}\lambda(t)g\lambda(t^{-1})$
exists, is a parabolic subgroup $P$ of $G$. Since $\lambda$ is also
a 1-ps of $\hG$, it also defines a parabolic subgroup $\hP$ of $\hG$.
Note that $P$ is contained in $\hP$, then we 
consider the immersion $\iota\,:\,G/P\longto \hG/\hP$
and the induced comorphism 
$$
\iota^*\,:\,H^*(\hG/\hP,\RR)\longto H^*(G/P,\RR)
$$
in cohomology.

Let $\Ta$ (resp. $\hTa$) denote the tangent space of $G/P$
(resp. $\hG/\hP$) at the point $P/P$ (resp. $\hP/\hP$).
We also denote  by $\iota$ the immersion of $\Ta$ in $\hTa$.

Let $W_P$ denote the Weyl group of $P$ and let $W^P$ be the set of minimal 
length representatives of the cosets of $W/W_P$.
Let $w\in W^P$.
Set $\Lambda_w=\overline{w^{-1}BwP/P}$ and $\Ta_w=T_{P/P}\Lambda_w$. For $\hw\in \hW^\hP$, we define as before $\Lambda_\hw\subset \hG/\hP$ and $\hTa_\hw$.
We assume that 
\begin{eqnarray}
  \label{eq:1}
  \codim(\Lambda_w,G/P)+\codim(\Lambda_\hw,\hG/\hP)=\dim(G/P).
\end{eqnarray}

\begin{defi}
The pair $(w,\hw)$ is said to be {\it Levi-movable} if there exists
$\hl\in\hL$ such that
\begin{eqnarray}
  \label{eq:2}
\iota(\Ta_w)\cap \hl\hTa_\hw=\{0\}.
\end{eqnarray}
\end{defi}

Let $\sigma_w\in  H^*(G/P,\RR)$ (resp. $\sigma_\hw\in H^*(\hG/\hP,\RR)$) 
denote the cohomology class of $\Lambda_w$ (resp. $\Lambda_\hw$).
Let $[pt]$ denote the class of the point in $H^*(G/P,\RR)$.
An important consequence of Levi-movability of $(w,\hw)$ is the following
nonvanishing:
$$
\iota^*(\sigma_\hw).\sigma_w=c[pt]\mbox{ for some positive integer }c.
$$

The action of $\lambda$ induces decompositions 
$$
\Ta=\bigoplus_{k<0}\Ta^k,\quad \mbox{and} \quad \hTa=\bigoplus_{k<0}\hTa^k;
$$
and 
$$
\Ta_w=\bigoplus_{k<0}\Ta^k_w,\quad \mbox{and}  \quad\hTa_\hw=\bigoplus_{k<0}\hTa^k_\hw.
$$
The following result is a useful observation.

\begin{lem}
 \label{lem:Lmovdecompose}
The pair  $(w,\hw)$ is  Levi-movable if and only if 
$$
\forall k\in\ZZ_{<0}\quad\exists \hl\in\hL\quad \iota(\Ta_w^k)\cap \hl\hTa_\hw^k=\{0\}.
$$
In particular, if $(w,\hw)$ is  Levi-movable then 
\begin{eqnarray}
  \label{eq:gdimcond}
  \forall k\in\ZZ_{<0}\quad \dim(\Ta_w^k)+\dim(\hTa_\hw^k)=\dim(\hTa^k).
\end{eqnarray}
\end{lem}

\begin{proof}
Since the actions of $\lambda$ and $\hL$ commute, the pair  $(w,\hw)$ is  Levi-movable if and only if 
$$
\exists \hl\in\hL\quad\forall k\in\ZZ_{<0}\quad \iota(\Ta_w^k)\cap \hl\hTa_\hw^k=\{0\}.
$$
But 
the condition $\iota(\Ta_w^k)\cap \hl\hTa_\hw^k=\{0\}$ is open
in $\hl$. This allows to permute the ``$\exists$'' and the ``$\forall$''.
\end{proof}

Denote by $\Phi$ the set of roots of $(G,T)$ and consider the root space decomposition of 
$\lg=\oplus_{\alpha\in\Phi}\lg_\alpha\oplus\lt$.
Let $\Phi^+$ be the set of positive roots of $B$ and set $\Phi^-=-\Phi^+$.
Consider the natural pairing $\langle\ ,\ \rangle$ between 1-ps and 
characters of $T$. Observe that
$\Ta^k$ is canonically isomorphic to 
$$\bigoplus_{
    \alpha\in\Phi,\,
\langle\lambda ,\alpha \rangle=k
}\lg_\alpha.
$$
Denote by $\Phi^k$ the set of 
$\alpha\in\Phi$ such that $\langle\lambda ,\alpha \rangle=k$.
The space $\Ta_w^k$ is canonically isomorphic to 
$$\bigoplus_{
    \alpha\in\Phi(w),\,
\langle\lambda ,\alpha \rangle=-k
}
\lg_{-\alpha},
$$
where $\Phi(w)=\Phi^+\cap w^{-1}\Phi^-$. Denote by $\Phi(w)^k$ the set of 
$\alpha\in\Phi(w)$ such that $\langle\lambda ,\alpha \rangle=-k$.\\

\subsection{Description of the cone $\QQ_{\geq 0}LR(G,\hG)$}

Recall that $\Wt_T(\hlg/\lg)$ is the set of non trivial weights of $T$ in $\hlg/\lg$. 
Let $\Chi(T)\otimes_\ZZ \QQ$ denote the rational vector space  spanned by the
characters of $T$.
We consider the set of  hyperplanes $H$ of $\Chi(T)\otimes_\ZZ \QQ$ spanned by some
elements of $\Wt_T(\hlg/\lg)$.
For each such hyperplane $H$ there exist exactly two opposite indivisible 1-ps 
$\pm\lambda_H$ that are orthogonal (for the paring $\langle\cdot,\cdot\rangle$) to $H$.
The so obtained 1-ps form a stable set under the action of $W$.
Let $\{\lambda_1,\,\dots,\lambda_n\}$ be the set of dominant such 1-ps.   
Denote by $P_i$ and $\hP_i$ the parabolic subgroups of $G$ and $\hG$ 
associated to $\lambda_i$.
A 1-ps of $T$ is said to be {\it admissible} if the hyperplane
of $\Chi(T)\otimes_\ZZ \QQ$ defined by 
$\langle\lambda,\cdot\rangle=0$ is spanned by some elements of $\Wt_T(\hlg/\lg)$, 
or equivalently if $\lambda$ belongs to some $\ZZ_{>0}W\lambda_i$.

\begin{theo}\label{th:desccone}\cite{GITEigen} 
(see also \cite[Proposition 2.3]{RR})
Suppose that every connected, closed and normal subgroup of $\hG$ contained 
in $G$ is trivial.   
Then $\QQ_{\geq 0}LR(G,\hG)$ has non empty interior  in $\Chi(T\times\hT)\otimes_\ZZ \QQ$.

\begin{enumerate}[(i)]
\item 
Let $i\in\{1,\dots,n\}$ and let  $(w,\hw)\in W^{P_i}\times\hW^{\hP_i}$
be a  Levi-movable pair. Then for any $(\nu,\hnu)$ in $\QQ_{\geq 0}LR(G,\hG)$
we have
\begin{eqnarray}
  \label{eq:ineg}
\langle w\lambda_i, \nu \rangle+\langle\hw\lambda_i,\hnu\rangle  \leq 0.
\end{eqnarray}
\item 
\label{ass2}
A dominant weight $(\nu,\hnu)$ belongs to  $\QQ_{\geq 0}LR(G,\hG)$ 
if and only if 
\begin{eqnarray}
  \label{eq:ineg2}
\langle w\lambda_i, \nu \rangle+\langle\hw\lambda_i,\hnu\rangle  \leq 0.
\end{eqnarray}
for any $i=1,\dots, n$ 
and for any Levi-movable pair $(w,\hw)\in W^{P_i}\times\hW^{\hP_i}$ such that 
$\iota^*(\sigma_\hw)\cdot\sigma_w=[pt]\in {\rm H}^*(G/P(\lambda_i),\ZZ)$.
\item 
Each inequality~\eqref{eq:ineg2} in assertion~\eqref{ass2}  corresponds to a codimension one face
of the cone $\QQ_{\geq 0}LR(G,\hG)$.
\end{enumerate}
\end{theo}

\subsection{Finalization of the method}\label{sec:algo}

To decide if a given pair $(G,\hG)$ has the saturation property, we first
compute the cone $\QQ_{\geq 0}LR(G,\hG)$ following the steps below.
 
\begin{enumerate}[Step~1.]
\item Compute the weights of $T$ in $\hlg/\lg$ and the admissible 1-ps 
$\lambda_1,\dots,\lambda_n$.
\item For each $i$ and each $w\in W^{P_i}$ compute $\Phi(w)^k$ for each $k$.
Similarly compute the subsets    $\Phi(\hw)^k$.
\item List for each $i$, the set of pairs 
$(w,\hw)\in W^{P_i}\times\hW^{\hP_i}$ satisfying condition~\eqref{eq:gdimcond}.
\item For each pair $(w,\hw)$ in this list, find an $\hl$ such that the condition~\eqref{eq:2} is satisfied. It may happen that we do not find such a $\hl$, but it does not mean necessarily that it does not exist. In that case, to be sure that $(w,\hw)$ is not Levi-movable, we have to go until Step~6 and come back to this step if necessary. (Note that, since the set of $\hl$ satisfying  condition~\eqref{eq:2} is open in $\hL$, the probability to have the good result at the first time is close to 1.)

The L-movable pairs $(w,\hw)$ we found at this step, give a list of inequalities~\eqref{eq:ineg} satisfied by the points of $\QQ_{\geq 0}LR(G,\hG)$ and then define a cone ${\mathcal C}$ containing $\QQ_{\geq 0}LR(G,\hG)$.

\item Compute the rays of ${\mathcal C}$.
\item Check that each ray belongs to  $\QQ_{\geq 0}LR(G,\hG)$.
If it is true, then we deduce that $\QQ_{\geq 0}LR(G,\hG)\subset{\mathcal C}$. If one of the rays does not belong to $\QQ_{\geq 0}LR(G,\hG)$, we have to come back to Step~4 and to find an L-movable pair more.
\end{enumerate}

At this point, we can also compute the redundant inequalities, by computing the rays of the dual cone of ${\mathcal C}$. We proceed as follows with 4ti2. We take the rays of ${\mathcal C}$ as inequalities to get ${\mathcal C}^\vee$, and we compute the rays of ${\mathcal C}^\vee$, which give the minimal set of inequalities defining ${\mathcal C}$.

\bigskip
Now to decide if the pair $(G,\hG)$ has the saturation property it is 
sufficient to

\begin{enumerate}
\item Compute the Hilbert bases of the semigroup 
 $\QQ_{\geq 0}LR(G,\hG)\cap\ZZ LR(G,\hG)$.
\item Check whether or not the elements of the Hilbert bases belong to
 $LR(G,\hG)$.
\end{enumerate}

{\bf Notation:} In all examples, we take Bourbaki's notation for simple roots, simple reflections, fundamental weights, and $\epsilon_i$'s, adding a hat to data corresponding to $\hG$.

\section{A first example with details: $\SL_3$ in $G_2$}

The root system of $G_2$ is generally represented by the following picture.

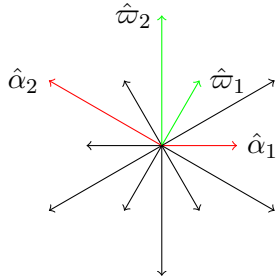
\begin{center}
\begin{tikzpicture}[scale=1]
  \path (0,0) coordinate (origin);
\coordinate [label=right:$\halpha_1$] (A1) at (0:1cm);
\coordinate [label=right:$\hat{\varpi}_1$] (A2) at (1*60:1cm);
\path (2*60:1cm) coordinate (A3);
\path (3*60:1cm) coordinate (A4);
\path (4*60:1cm) coordinate (A5);
\path (5*60:1cm) coordinate (A6);

\draw [->,color=red] (origin) -- (A1);
\draw [->,color=green]  (origin) -- (A2) ;\draw [->] (origin) -- (A3) ;\draw [->] (origin) -- (A4) ;\draw [->] (origin) -- (A5) ;\draw [->] (origin) -- (A6);

\coordinate [label=left:$\hat{\varpi}_2$]  (B1) at (90:1.732) ;
\coordinate [label=left:$\halpha_2$] (B2) at (1*60+90:1.732cm);
\path (2*60+90:1.732cm) coordinate (B3);
\path (3*60+90:1.732cm) coordinate (B4);
\path (4*60+90:1.732cm) coordinate (B5);
\path (5*60+90:1.732cm) coordinate (B6);

\draw [->,color=green] (origin) -- (B1);
\draw [->,color=red]  (origin) -- (B2) ;
\draw [->] (origin) -- (B3) ;\draw [->] (origin) -- (B4) ;\draw [->] (origin) -- (B5) ;\draw [->] (origin) -- (B6);
\end{tikzpicture}
\end{center}

The set of long roots of $G_2$ gives a subsystem of roots of type $A_2$.
We follow the steps of Section~\ref{sec:algo}.

\begin{enumerate}[Step~1.]

\item The weights of $T$ on
$\hlg/\lg$ are the short roots packed in two opposite triangles that are stable by the Weyl group $W$ generated by the reflections associated to long roots.
There is exactly one indivisible dominant admissible 1-ps $\lambda$ defined by:
$$
\lambda(t)=\diag(t,1,t^{-1}).
$$

\item The variety $G/P(\lambda)$ is the complete flag variety $\Fl(\CC^3)$.
Moreover $\hP(\lambda)$ is the maximal parabolic subgroup associated
to the long simple root; and $\hG/\hP(\lambda)$ is $Q^5$. 
The weights $\Wt_T(\hTa)$ of $\hT=T$ on $\hTa$ are  the five negative roots different 
from $-\halpha_1$. The set  $\Phi(\hTa)$ is represented by
\begin{center}
 \begin{tikzpicture}[scale=.2]
   \draw (0,0) -- (4,0) -- (4,1) -- (0,1) -- cycle;
   \draw (1,0) -- (1,1) (2,0) -- (2,1) (3,0) -- (3,1);
 \draw (1.5,0) -- (2.5,0) -- (2.5,-1) -- (1.5,-1) -- cycle;
 \end{tikzpicture}
 \end{center}
 where each box corresponds to a root in a canonical way.
For any $\hw\in \hW^{\hP}$, the opposite of the elements of $\Phi(\hw)$ are
 contained in $\Phi(\hTa)$ and represented by black boxes.

The weights of $\Ta$ are
the 3  long roots in $\Phi(\hTa)$ represented by the three corresponding boxes:
\begin{center}
  \begin{tikzpicture}[scale=0.2]
      \videGB
    \end{tikzpicture}
\end{center}
The 6 inversion sets $\Phi(w)$ for $w\in W^P$ and the 6  inversion sets 
$\Phi(\hw)$ for $w\in W^P$
are represented on Figure~\ref{fig:invSL3G2}.
 
\bigskip
\begin{figure}

\begin{tabular}{c}
\begin{tikzpicture}
  \node [rectangle] (a) at (0,0) {
    \begin{tikzpicture}[scale=0.2]
      \videGB
    \end{tikzpicture}
};

  \node [rectangle] (b) at (2,-1) {\begin{tikzpicture}[scale=0.2]
\grisun\videGB
\end{tikzpicture}};

\node [rectangle] (c) at (2,1) {\begin{tikzpicture}[scale=0.2]
\grisdeux\videGB
\end{tikzpicture}};

\node [rectangle] (d) at (4,-1) {\begin{tikzpicture}[scale=0.2]
\grisun\gristrois\videGB
\end{tikzpicture}};

\node [rectangle] (e) at (4,1) {\begin{tikzpicture}[scale=0.2]
\grisdeux\gristrois\videGB
\end{tikzpicture}};

\node [rectangle] (f) at (6,0) {\begin{tikzpicture}[scale=0.2]
\grisdeux\gristrois\grisun\videGB
\end{tikzpicture}};

\draw (a) -- node[below] {$s_1$} (b);
\draw (a) -- node[above] {$s_2$} (c);
\draw (b) -- node[below] {$s_2$} (d);
\draw (c) -- node[above] {$s_1$} (e);
\draw (e) -- node[above] {$s_2$} (f);
\draw (d) -- node[below] {$s_1$} (f);
\end{tikzpicture}\\

\begin{tikzpicture}
  \node [rectangle] (a) at (0,0) {\begin{tikzpicture}[scale=0.2]\draw (0,0) -- (4,0) -- (4,1) -- (0,1) -- cycle;
  \draw (1,0) -- (1,1) (2,0) -- (2,1) (3,0) -- (3,1);
\draw (1.5,0) -- (2.5,0) -- (2.5,-1) -- (1.5,-1) -- cycle;\end{tikzpicture}};

   \node [rectangle] (b) at (2,0) {\begin{tikzpicture}[scale=0.2]
\draw  [fill=gray!60] (3,0) -- (4,0) -- (4,1) -- (3,1) -- cycle;
\draw (0,0) -- (4,0) -- (4,1) -- (0,1) -- cycle;
  \draw (1,0) -- (1,1) (2,0) -- (2,1) (3,0) -- (3,1);
\draw (1.5,0) -- (2.5,0) -- (2.5,-1) -- (1.5,-1) -- cycle;\end{tikzpicture}};

 \node [rectangle] (c) at (4,0) {\begin{tikzpicture}[scale=0.2]
\draw  [fill=gray!60] (2,0) -- (4,0) -- (4,1) -- (2,1) -- cycle;
\draw (0,0) -- (4,0) -- (4,1) -- (0,1) -- cycle;
  \draw (1,0) -- (1,1) (2,0) -- (2,1) (3,0) -- (3,1);
\draw  (1.5,0) -- (2.5,0) -- (2.5,-1) -- (1.5,-1) -- cycle;\end{tikzpicture}};

\node [rectangle] (d) at (6,0) {\begin{tikzpicture}[scale=0.2]
\draw  [fill=gray!60] (2,0) -- (4,0) -- (4,1) -- (2,1) -- cycle;
\draw (0,0) -- (4,0) -- (4,1) -- (0,1) -- cycle;
  \draw (1,0) -- (1,1) (2,0) -- (2,1) (3,0) -- (3,1);
\draw [fill=gray!60] (1.5,0) -- (2.5,0) -- (2.5,-1) -- (1.5,-1) -- cycle;\end{tikzpicture}};

 \node [rectangle] (e) at (8,0) {\begin{tikzpicture}[scale=0.2]\draw (0,0) -- (4,0) -- (4,1) -- (0,1) -- cycle;
\draw  [fill=gray!60] (1,0) -- (4,0) -- (4,1) -- (1,1) -- cycle;
  \draw (1,0) -- (1,1) (2,0) -- (2,1) (3,0) -- (3,1);
\draw [fill=gray!60](1.5,0) -- (2.5,0) -- (2.5,-1) -- (1.5,-1) -- cycle;\end{tikzpicture}};

 \node [rectangle] (f) at (10,0) {\begin{tikzpicture}[scale=0.2]\draw[fill=gray!60] (0,0) -- (4,0) -- (4,1) -- (0,1) -- cycle;
  \draw (1,0) -- (1,1) (2,0) -- (2,1) (3,0) -- (3,1);
\draw[fill=gray!60](1.5,0) -- (2.5,0) -- (2.5,-1) -- (1.5,-1) --
cycle;\end{tikzpicture}};

\draw (a) -- node[above] {$\hs_2$} (b);
\draw (b) -- node[above] {$\hs_1$} (c);
\draw (c) -- node[above] {$\hs_2$} (d);
\draw (d) -- node[above] {$\hs_1$} (e);
\draw (e) -- node[above] {$\hs_2$} (f);
\end{tikzpicture}
\end{tabular}
\caption{Inversion sets for $G/P$ and $\hG/\hP$}
\label{fig:invSL3G2}
\end{figure}
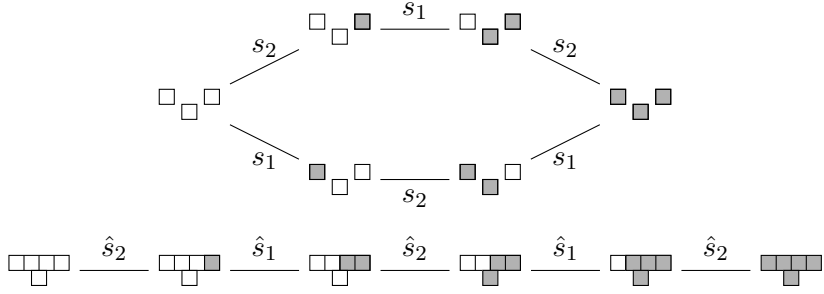

\item Only 4 pairs $(w,\,\hw)$ satisfy condition~\eqref{eq:gdimcond}:

\begin{tabular*}{1.0\linewidth}{cccc}
$\big ($ \begin{tikzpicture}[scale=0.2]
      \videGB
    \end{tikzpicture},\,
\begin{tikzpicture}[scale=0.2]\draw[fill=gray!60] (0,0) -- (4,0) -- (4,1) -- (0,1) -- cycle;
  \draw (1,0) -- (1,1) (2,0) -- (2,1) (3,0) -- (3,1);
\draw[fill=gray!60](1.5,0) -- (2.5,0) -- (2.5,-1) -- (1.5,-1) --
cycle;\end{tikzpicture}
$\big )$
&
$\big ($
\begin{tikzpicture}[scale=0.2]
\grisun\videGB
\end{tikzpicture},\,
\begin{tikzpicture}[scale=0.2]\draw (0,0) -- (4,0) -- (4,1) -- (0,1) -- cycle;
\draw  [fill=gray!60] (1,0) -- (4,0) -- (4,1) -- (1,1) -- cycle;
  \draw (1,0) -- (1,1) (2,0) -- (2,1) (3,0) -- (3,1);
\draw [fill=gray!60](1.5,0) -- (2.5,0) -- (2.5,-1) -- (1.5,-1) --
cycle;\end{tikzpicture}
$\big )$
&
$\big ($
\begin{tikzpicture}[scale=0.2]
\grisdeux\videGB
\end{tikzpicture},\,
\begin{tikzpicture}[scale=0.2]\draw (0,0) -- (4,0) -- (4,1) -- (0,1) -- cycle;
\draw  [fill=gray!60] (1,0) -- (4,0) -- (4,1) -- (1,1) -- cycle;
  \draw (1,0) -- (1,1) (2,0) -- (2,1) (3,0) -- (3,1);
\draw [fill=gray!60](1.5,0) -- (2.5,0) -- (2.5,-1) -- (1.5,-1) --
cycle;\end{tikzpicture}
$\big )$
&
$\big ($
\begin{tikzpicture}[scale=0.2]
\grisdeux\gristrois\grisun\videGB
\end{tikzpicture},\,
\begin{tikzpicture}[scale=0.2]
\draw  [fill=gray!60] (2,0) -- (4,0) -- (4,1) -- (2,1) -- cycle;
\draw (0,0) -- (4,0) -- (4,1) -- (0,1) -- cycle;
  \draw (1,0) -- (1,1) (2,0) -- (2,1) (3,0) -- (3,1);
\draw  (1.5,0) -- (2.5,0) -- (2.5,-1) -- (1.5,-1) --
cycle;\end{tikzpicture}
$\big )$
  
\end{tabular*}\\

\item
 The two first pairs are clearly L-movable (with $\hl$ equals the identity) 
 and the third pair is also L-movable (with $\hl=\hat{s}_1$). 

Consider the last pair $(w=s_1s_2s_1,\hw=\hs_1 \hs_2)$.
As a $\hL$-module, $\hTa^{-1}$ is isomorphic to the space of homogeneous 
polynomial function of degree 3 in 2 variables $x$ and $y$.
Then $\hTa_\hw^{-1}$ identify with the set of polynomial functions with $[0:1]$ 
as  double root.
There exists $\hl\in\hL$ such that  $\hl\hTa_\hw^{-1}$ identify with the set of 
polynomial functions with $[1:1]$ as  double root.
But $\Ta_w^{-1}$ identifies with the span of $x^3$ and $y^3$. 
Then $\hl\hTa_\hw^{-1}\cap \Ta_w^{-1}=\{0\}$.
Hence the pair is Levi-movable.

\bigskip
We set  $\nu=a\varpi_1+b\varpi_2$ and
$\hnu=A\hvarpi_1+B\hvarpi_2$.
The inequalities \eqref{eq:ineg} corresponding to the 4 Levi-movable 
pairs are
\begin{enumerate}
\item $B\leq a+b\leq A+2B$;
\item $\max(a,b)\leq A+B$;\\
to which we add the 4 dominancy inequalities
\item $0\leq\min(a,b,A,B)$.
\end{enumerate}

\item \label{raygen}
The extremal rays of the associated cone ${\mathcal C}$ 
are generated by the following 
pairs $(\nu,\hnu)$:
$(0,\hvarpi_1)$, $(\varpi_2,\hvarpi_1)$, $(\varpi_2,\hvarpi_2)$,
$(\varpi_1,\hvarpi_1)$, $(\varpi_1,\hvarpi_2)$ and
$(\varpi_1+\varpi_2,\hvarpi_2)$.

\item \label{CLR} The decompositions of the two fundamental representations of $G_2$ 
as $\SL_3$-module show the primitive generators of these 6 rays belong to $LR(\SL_3,G_2)$.
Then  ${\mathcal C}=\QQ_{\geq 0}LR(\SL_3,G_2)$.
\end{enumerate}

 Since $\hG$ has a trivial center, $\ZZ LR(\SL_3,G_2)$ is $X(T\times\hT)$.
Using 4ti2, we compute the Hilbert basis of 
$\ZZ LR(\SL_3,G_2)\cap \QQ_{\geq 0}LR(\SL_3,G_2)$.
It coincides with the list given at Step~\ref{raygen}.
Then Step~\ref{CLR} shows that $(\SL_3,G_2)$ has the saturation property.

\section{A second example with details: $G_2$ in $\Spin_7$}

The group $G=G_2$ has a simple representation of dimension 7 which induces
an embedding of $G_2$ in $\SO_7$. Since $G_2$ is simply connected,
$G_2$ is also embedded in $\hG=\Spin_7$. 

\begin{enumerate}[Step~1.]
\item 
As a $G_2$-module $\so_7={\rm Lie}(\Spin_7)$ is isomorphic to 
$\lg_2\oplus V_G(\varpi_1)$. The non-zero weights of $V_G(\varpi_1)$ are the 6 short roots of $G_2$, then there is a unique indivisible dominant admissible 
1-ps $\lambda$ defined by $\langle\lambda,\alpha_1\rangle=0$ and $\langle\lambda,\alpha_2\rangle=1$.
Set $P=P(\lambda)$ and $\hP=\hP(\lambda)$.

\item

The homogeneous space $G/P$ is the quadric $Q^5$. The inversion sets $\Phi(w)$ 
for $w\in W^P$ are already represented in Figure~\ref{fig:invSL3G2}.

Let $\rho\,:\, X(\hT)\longto X(T)$ denote the restriction map.
It satisfies $\rho(\halpha_1)=\rho(\halpha_3)=\alpha_1$ and 
$\rho(\halpha_2)=\alpha_2$.
This allows to compute $\langle\lambda,\halpha_i\rangle$ for $i=1,2$
and $3$. We deduce that in the dual basis of $(\hat{\varepsilon}_i)_{i=1,2,3}$,
$\lambda=(1,1,0)$ (as a 1-ps in $\hT$).
In particular $\hG/\hP=Gr_Q(2,7)$ and the inversion sets for $\hG/\hP$ are represented by the following diagrams, where boxes correspond from top to bottom and left to right to the weights $\hat{\varepsilon}_1-\hat{\varepsilon}_3,\,\hat{\varepsilon}_2-\hat{\varepsilon}_3,\,\hat{\varepsilon}_1,\,\hat{\varepsilon}_2,\,\hat{\varepsilon}_1+\hat{\varepsilon}_3,\,\hat{\varepsilon}_2+\hat{\varepsilon}_3$ and $\hat{\varepsilon}_1+\hat{\varepsilon}_2$. We describe the elements of $\hW$ by the permutation acting on a basis of $V_\hG(\hvarpi_1)$ consisting of $\hat U$-stable vectors on which $\hT$ acts with weights (in this order) $\hat{\varepsilon}_1,\,\hat{\varepsilon}_2,\,\hat{\varepsilon}_3,\,0,\,-\hat{\varepsilon}_3,\,-\hat{\varepsilon}_2$ and $-\hat{\varepsilon}_1$.

$$
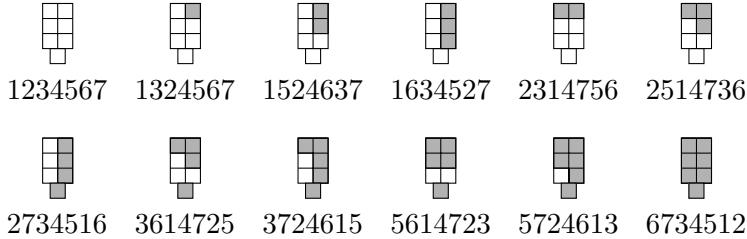
\begin{array}{cccccc}
  \iundeux&\iuntrois&\iuncinq&\iunsix&\ideuxtrois&\ideuxcinq\\
1234567&1324567&1524637&1634527&2314756&2514736\\[1em]
\ideuxsept & \itroissix & \itroissept & \icinqsix & \icinqsept & \isixsept\\
2734516    & 3614725  & 3724615    & 5614723 & 5724613  & 6734512
\end{array}
$$

\item Only 8 pairs $(w,\hw)$ satisfy condition~\eqref{eq:gdimcond}. We give them in the table bellow, with the data that give the corresponding inequalities. Set $\nu=a\hvarpi_1+b\hvarpi_2$ and
$\hat \nu=A\hat \varepsilon_1+B\hat \varepsilon_1+C\hat \varepsilon_1$.

$$
\begin{array}{|c|c|c|c|c|c|}
  \hline
w&\Phi(w)&\langle w\lambda,\nu\rangle&
\hat w&\Phi(\hat w)&\langle \hat w\lambda,\hat \nu\rangle\\
\hline
e&\begin{tikzpicture}[scale=0.2]\draw (0,0) -- (4,0) -- (4,1) -- (0,1) -- cycle;
  \draw (1,0) -- (1,1) (2,0) -- (2,1) (3,0) -- (3,1);
\draw (1.5,0) -- (2.5,0) -- (2.5,-1) -- (1.5,-1) --
cycle;\end{tikzpicture}
&2b+a&6734512&\isixsept&-A-B\\
\hline
s_\beta&
\begin{tikzpicture}[scale=0.2]
\draw  [fill=gray!60] (3,0) -- (4,0) -- (4,1) -- (3,1) -- cycle;
\draw (0,0) -- (4,0) -- (4,1) -- (0,1) -- cycle;
  \draw (1,0) -- (1,1) (2,0) -- (2,1) (3,0) -- (3,1);
\draw (1.5,0) -- (2.5,0) -- (2.5,-1) -- (1.5,-1) --
cycle;\end{tikzpicture}&a+b&
5724613 &\icinqsept&-A-C\\
\hline
\multirow{2}{*}{$s_\alpha s_\beta$}&
\multirow{2}{*}{\begin{tikzpicture}[scale=0.2]
\draw  [fill=gray!60] (2,0) -- (4,0) -- (4,1) -- (2,1) -- cycle;
\draw (0,0) -- (4,0) -- (4,1) -- (0,1) -- cycle;
  \draw (1,0) -- (1,1) (2,0) -- (2,1) (3,0) -- (3,1);
\draw  (1.5,0) -- (2.5,0) -- (2.5,-1) -- (1.5,-1) -- cycle;\end{tikzpicture}}
&\multirow{2}{*}{b}&3724615&\itroissept&-A+C\\
\cline{4-6}
&&&5614723&\icinqsix&-B-C\\
\hline
\multirow{2}{*}{$s_\alpha s_\beta s_\alpha s_\beta$}&
\multirow{2}{*}{\begin{tikzpicture}[scale=0.2]
\draw (0,0) -- (4,0) -- (4,1) -- (0,1) -- cycle;
\draw  [fill=gray!60] (1,0) -- (4,0) -- (4,1) -- (1,1) -- cycle;
  \draw (1,0) -- (1,1) (2,0) -- (2,1) (3,0) -- (3,1);
\draw [fill=gray!60](1.5,0) -- (2.5,0) -- (2.5,-1) -- (1.5,-1) -- cycle;
\end{tikzpicture}}
&\multirow{2}{*}{-a-b}&1634527&\iunsix&A-B\\
\cline{4-6}
&&&2514736&\ideuxcinq&B-C\\
\hline

\multirow{2}{*}{$s_\beta s_\alpha s_\beta s_\alpha s_\beta$}&
\multirow{2}{*}{\begin{tikzpicture}[scale=0.2]\draw[fill=gray!60] (0,0) -- (4,0) -- (4,1) -- (0,1) -- cycle;
  \draw (1,0) -- (1,1) (2,0) -- (2,1) (3,0) -- (3,1);
\draw[fill=gray!60](1.5,0) -- (2.5,0) -- (2.5,-1) -- (1.5,-1) -- cycle;\end{tikzpicture}}
&\multirow{2}{*}{-a-2b}&1524637&\iuncinq&A-C\\
\cline{4-6}
&&&2314756&\ideuxtrois&B+C\\
\hline
\end{array}
$$

\item The semi-simple part of the Levi subgroup $\hL$ is isomorphic to $\SL(2)\times\SL(2)$. With, for example, $$
\hl=\left(
\left( 
\begin{array}{cc}
  1&3\\
1&4
\end{array}
\right ), 
\left( 
\begin{array}{cc}
  2&1\\
3&2
\end{array}
\right )\right),
$$

we obtain that the 7 first pairs $(w,\hw)$ in the table are L-movable.

\item The inequalities \eqref{eq:ineg} corresponding to the 7 Levi-movable 
pairs are
\begin{enumerate}
\item $a\geq 0$, $b\geq 0$;
\item $A\geq B\geq C\geq 0$;
\item $A-C\leq 2b+a\leq A+B$;
\item $\max(B-C,\,A-B)\leq a+b\leq A+C$;
\item $b\leq \min(B+C,\,A-C)$;

to which we add the 5 dominancy inequalities
\item $a,b\geq 0$;
\item $A\geq B\geq C\geq 0$.
\end{enumerate}

\item The 7 extremal rays of the associated cone $\mathcal{C}$ are generated by the following pairs $(\nu,\hnu)$: $(\varpi_1,\hvarpi_1)$, $(\varpi_1,\hvarpi_2)$, $(\varpi_2,\hvarpi_2)$, $(0,\hvarpi_3)$, $(\varpi_1,\hvarpi_3)$, $(\varpi_2,\hvarpi_1+\hvarpi_3)$ and $(\varpi_2,\hvarpi_1+\hvarpi_2)$.

\item We can check that all these 7 pairs $(\nu,\hnu)$ are in $LR(G_2,\Spin_7)$ and then $\mathcal{C}=LR(G_2,\Spin_7)$.
We could also remark that the inequality corresponding to the last pair of the table is not satisfied (because $(\varpi_1,\hvarpi_1+2\hvarpi_3)=(\varpi_1,2\hat{\varepsilon_1}+\hat{\varepsilon}_2+\hat{\varepsilon}_3)\in LR(G_2,\Spin_7)$), so that the last pair of the table is not $L$-movable.
\end{enumerate}

Since $G$ has a trivial center, $\ZZ LR(G_2,\Spin_7)$ is $X(T\times\hT)$.
Using 4ti2, we compute the Hilbert basis of 
$\ZZ LR(G_2,\Spin_7)\cap \QQ_{\geq 0}LR(G_2,\Spin_7)$.
It coincides with the list given at Step~\ref{raygen}.
Then Step~\ref{CLR} shows that $(G_2,\Spin_7)$ has the saturation property.

\begin{rem}
  Let $\hT_{\SO}$ be the maximal torus of $\SO_7$. 
Then $LR(G_2,\SO_7)=X(T\times \hT_{\SO})\cap LR(G_2,\Spin_7)$.
In particular $(G_2,\SO_7)$ has the saturation property.
Observe that the Hilbert basis of $LR(G_2,\SO_7)$ is the union of the 7 
primitive generators of the extremal rays and the following 3 pairs: 
$(\varpi_1,2\hvarpi_3)$, $(\varpi_2,\hvarpi_1+2\hvarpi_3)$ and $(\varpi_1+\varpi_2,\hvarpi_1+2\hvarpi_3)$.
\end{rem}

>From the remaining examples of this paper, we use computations with Sage in order to get the Levi-movable pairs, 4ti2 to compute the Hilbert basis and Sage to check the saturation. All the programs used to obtain the results below are available in authors' web pages.

\section{$B_4$ in $F_4$}

A more detailed version of this section (using only few computations with computer) can be found in authors' web pages. 

The root system $\hat{\Phi}$ of $F_4$ contains 24 short roots
$$
\pm\hvarepsilon_i
\quad \frac 1 2 (\pm \hvarepsilon_1 \pm \hvarepsilon_2 \pm \hvarepsilon_3 \pm \hvarepsilon_4)
$$
and 24 long roots
$$
\pm \hvarepsilon_i \pm \hvarepsilon_j\quad i<j.
$$

There are 3 ways to embed $\Spin_9$ in $F_4$, they are all equivalent up to the action of $\hW$. 
We choose the one where $\Phi$ consists of the long roots of $\hat \Phi$ and the 8 short roots $\pm\hvarepsilon_i$ with $i=1,\,2,\,3$ and $4$. Note that $\varepsilon_i=\hat \varepsilon_i$.
Then, the simple roots of $B_4$ are 
$$
\alpha_1=2\halpha_4+\halpha_2+2\halpha_3,
\quad
\alpha_2=\halpha_1,
\quad
\alpha_3=\halpha_2,
\quad
\alpha_4=\halpha_3,
$$

A 1-ps
$\lambda=a\varepsilon_1^*+b\varepsilon_2^*+c\varepsilon_3^*+d\varepsilon_4^*$
is dominant if 
$
a\geq b\geq c\geq d\geq 0.
$
The weights of $T=\hT$ in $\hlg/\lg$ are
$
 \frac 1 2 (\pm \hvarepsilon_1 \pm \hvarepsilon_2 \pm \hvarepsilon_3 \pm
 \hvarepsilon_4).
$
The Weyl group $W$ of $B_4$ is $S_4.(\ZZ/2\ZZ)^4$, acting on the weights above in a natural way.
We deduce that there are two dominant indivisible admissible 1-ps:
$$
\lambda_1=\varepsilon_1^*+\varepsilon_2^*+\varepsilon_3^*+
\varepsilon_4^*\quad\mbox{ and }\quad
\lambda_2=\varepsilon_1^*+\varepsilon_2^*.
$$

To check the L-movability of the pairs, we need to know the following facts.
\begin{enumerate}
\item For $\lambda_1$, the Levi subgroup  $\hL$ is of type $B_3$ and the two tangent spaces $\hTa^{-1}$ and $\hTa^{-2}$ are isomorphic  to the spinorial representation and the standard representation as a $\Spin_7$-module.
\item For $\lambda_2$, the Levi subgroup  $\hL$ is of type $C_3$ and the two tangent spaces $\hTa^{-1}$ and $\hTa^{-2}$ are isomorphic to the third fundamental representation (subrepresentation of $\bigwedge^3\CC^6$) and the trivial representation as a  $\Spin_7$-module.
\end{enumerate}

Then, the Sage programs (and also 4ti2 to compute the rays and the Hilbert basis as in the previous sections) give the following result.

They are 36 (6 for $\lambda_1$ and 30 for $\lambda_2$) pairs satisfying condition~\eqref{eq:gdimcond} that give 28 Levi-movable pairs.
The cone $\QQ_{\geq 0}LR(\Spin(9),F_4)$ is defined by 36 non-redundant inequalities (including the 8 dominancy inequalities), it has 20 rays whose primitive elements give the Hilbert basis of the cone. In the bases of fundamental weights, these elements are:
$$\begin{array}{cccccccc@{\qquad}|@{\qquad }cccccccc} 
 0 & 0 & 0 & 0 & 0 & 0 & 0 & 1 & 0 & 1 & 0 & 0 & 1 & 0 & 0 & 0\\
 0 & 0 & 0 & 1 & 0 & 0 & 0 & 1 & 0 & 1 & 0 & 1 & 0 & 1 & 0 & 0 \\
 0 & 0 & 0 & 1 & 0 & 0 & 1 & 0 & 0 & 1 & 1 & 0 & 0 & 1 & 0 & 1 \\
 0 & 0 & 0 & 1 & 1 & 0 & 0 & 0 & 1 & 0 & 0 & 0 & 0 & 0 & 0 & 1 \\
 0 & 0 & 1 & 0 & 0 & 0 & 1 & 0 & 1 & 0 & 0 & 0 & 0 & 0 & 1 & 0 \\
 0 & 0 & 1 & 0 & 0 & 1 & 0 & 0 & 1 & 0 & 0 & 1 & 0 & 0 & 1 & 0 \\

 0 & 0 & 1 & 0 & 1 & 0 & 0 & 1 & 1 & 0 & 0 & 1 & 0 & 1 & 0 & 0 \\

 0 & 0 & 1 & 0 & 1 & 0 & 1 & 0 & 1 & 0 & 1 & 0 & 0 & 1 & 0 & 0 \\

 0 & 1 & 0 & 0 & 0 & 0 & 1 & 0 & 1 & 0 & 1 & 0 & 1 & 0 & 1 & 0\\

 0 & 1 & 0 & 0 & 0 & 1 & 0 & 0 & 1 & 0 & 1 & 0 & 1 & 1 & 0 & 0\\
\end{array}$$

We check easily, using Sage, that the pair $(B_4,F_4)$ has the saturation property.

\section{$F_4$ in $E_6$}

\noindent{\bf admissible 1-ps.}
The group $E_6$ has dimension $78$ and $F_4$ has dimension $52$. Hence
$\hlg/\lg$ has dimension $26$ and then it is the smallest representation $V_{F_4}(\varpi_4)$ of $F_4$.
But $\varpi_4=\varepsilon_1$ is a short root.
Hence $Wt_T(V_{\varpi_4})$ is the set 
of  24 short roots of $F_4$.
The hyperplanes spanned by short roots are the Levi subgroups
containing $T$ of semisimple rank 3 in $D_4$. Up to the Weyl group $W(D_4)$ of $D_4$, they
correspond bijectively with the simple roots of $D_4$. Then, up to $W$, there
are two dominant indivisible admissible 1-ps:

$$
\lambda_1=\varepsilon_1^*\quad\mbox{ and }\quad \lambda_2=\varepsilon_1^*+\varepsilon_2^*.
$$

To check the L-movability of the pairs, we need to know the following facts:
\begin{enumerate}
\item For $\lambda_1$, the Levi subgroup $\hL$ is of type $D_4$ and the two tangent spaces $\hTa^{-1}$ and $\hTa^{-2}$ are isomorphic to the direct sum of the two spinorial representations and the standard representation as a  $\Spin_8$-module.
\item For $\lambda_2$, the Levi subgroup  $\hL$ is of type $A_5$ and the two tangent spaces $\hTa^{-1}$ and $\hTa^{-2}$  are isomorphic  to the third fundamental representation $\bigwedge^3\CC^6$ and the trivial representation as a  $\SL_6$-module.
\end{enumerate}

Then, the Sage programs (and also 4ti2) give the following result.

The cone $\QQ_{\geq 0}LR(F_4,E_6)$ is defined by 61 non-redundant inequalities (including 10 dominancy inequalities), it has 37 rays whose primitive elements give the Hilbert basis of the cone. In the fundamental bases, these elements are:
$$\begin{array}{ccccccccccc|ccccccccccc}
0 & 0 & 0 & 0 & 0 & 0 & 0 & 0 & 0 & 1 &&
 0 & 0 & 0 & 0 & 1 & 0 & 0 & 0 & 0 & 0 \\
 0 & 0 & 0 & 1 & 0 & 0 & 0 & 0 & 0 & 1 &&
 0 & 0 & 0 & 1 & 0 & 0 & 0 & 0 & 1 & 0 \\
 0 & 0 & 0 & 1 & 0 & 0 & 1 & 0 & 0 & 0 &&
 0 & 0 & 0 & 1 & 0 & 1 & 0 & 0 & 0 & 0 \\
 0 & 0 & 0 & 1 & 1 & 0 & 0 & 0 & 0 & 0  && 
 0 & 0 & 1 & 0 & 0 & 0 & 0 & 0 & 1 & 0 \\
 0 & 0 & 1 & 0 & 0 & 0 & 0 & 1 & 0 & 0  & &
 0 & 0 & 1 & 0 & 0 & 0 & 1 & 0 & 0 & 0 \\
 0 & 0 & 1 & 0 & 0 & 0 & 1 & 0 & 1 & 0  & &
 0 & 0 & 1 & 0 & 0 & 1 & 0 & 0 & 0 & 1 \\
 0 & 0 & 1 & 0 & 0 & 1 & 0 & 0 & 1 & 0  & &
 0 & 0 & 1 & 0 & 0 & 1 & 1 & 0 & 0 & 0 \\
 0 & 0 & 1 & 0 & 1 & 0 & 0 & 0 & 0 & 1  & &
 0 & 0 & 1 & 0 & 1 & 1 & 0 & 0 & 0 & 0 \\
 0 & 1 & 0 & 0 & 0 & 0 & 0 & 1 & 0 & 0  & &
 0 & 1 & 0 & 0 & 0 & 0 & 0 & 1 & 1 & 0 &*\\
 0 & 1 & 0 & 0 & 0 & 0 & 1 & 0 & 0 & 1  & &
 0 & 1 & 0 & 0 & 0 & 0 & 1 & 0 & 1 & 0 &*\\
 0 & 1 & 0 & 0 & 0 & 0 & 1 & 1 & 0 & 0  & *&
 0 & 1 & 0 & 0 & 0 & 1 & 0 & 0 & 1 & 0 &*\\
 0 & 1 & 0 & 0 & 0 & 1 & 0 & 1 & 0 & 0  & *&
 0 & 1 & 0 & 0 & 0 & 1 & 1 & 0 & 0 & 0 &*\\
 0 & 1 & 0 & 0 & 0 & 1 & 1 & 0 & 1 & 0  & &
 0 & 1 & 0 & 0 & 1 & 0 & 0 & 0 & 1 & 0 \\
 0 & 1 & 0 & 0 & 1 & 1 & 0 & 0 & 0 & 1  & &
 0 & 1 & 0 & 1 & 0 & 0 & 0 & 2 & 0 & 0 &*\\
 0 & 1 & 0 & 1 & 0 & 0 & 1 & 0 & 1 & 0  & &
 1 & 0 & 0 & 0 & 0 & 0 & 0 & 0 & 1 & 0 \\
 1 & 0 & 0 & 0 & 0 & 0 & 0 & 1 & 0 & 0  & &
 1 & 0 & 0 & 0 & 0 & 0 & 1 & 0 & 0 & 0 \\
 1 & 0 & 0 & 0 & 0 & 1 & 0 & 0 & 0 & 0  & &
 1 & 0 & 0 & 1 & 0 & 0 & 0 & 1 & 0 & 0 \\
 1 & 0 & 1 & 0 & 0 & 0 & 0 & 1 & 0 & 1  & &
 1 & 0 & 1 & 0 & 1 & 0 & 0 & 1 & 0 & 0 \\
 1 & 1 & 0 & 0 & 1 & 0 & 0 & 1 & 0 & 1&
\end{array}$$

Among these 37 elements, 30 are given by the PRV Theorem (see \cite{MPR}). Moreover, the remaining 7 elements (with * in the list above) can be reduced to 5, by using the involution of $E_6$.
We now check these 5 elements, using Sage, to get the saturation property (see authors' web pages to get details).

\section{A family of examples: $\Sp_{2n}$ in $\SL_{2n}$}

Until $n=5$, Sage programs  (available in authors' web pages) and 4ti2 allow to prove the saturation property of the pair $(\Sp_{2n},\SL_{2n})$. 
In this section, we give the steps of Section~\ref{sec:algo} that we can do for any $n\geq 2$. And we give the results of computations for $n=2,\,3,\,4$ and $5$.

\subsection{Notation on the groups}

Let $V$ be a $2n$-dimensional vector space 
with basis $\base=(e_1,\dots,e_{2n})$.
Consider the bilinear symplectic form $\omega_n$ on 
$V$ with matrix 
\begin{eqnarray}
  \label{eq:defJn}
  \omega_n=\left(
  \begin{array}{cc}
   0 &J_n\\
-J_n&0
  \end{array}
\right), \quad{\rm where}\quad
J_n=\left(
  \begin{array}{c@{}c@{}c}
    &&1\\[-2pt]
&\revddots\\[-4pt]
1
  \end{array}
\right).
\end{eqnarray}

Let $G$ be the associated symplectic group.
Set $T=\{\diag(t_1,\dots,t_n,t_n^{-1},\dots,t_1^{-1})\,:\,t_i\in\CC^*\}$.
Let $B$ be the Borel subgroup of $G$ consisting of upper triangular matrices of $G$.

Here $\hG=\SL(V)$, $\hB$ is the subset of upper triangular matrices and $\hT$ is the subset of diagonal matrices.

For $i\in [1,n]$, let $\varepsilon_i$ denote the character of $T$ that maps 
$\diag(t_1,\dots,t_n,t_n^{-1},\dots,t_1^{-1})$ to $t_i$; then 
$\Chi(T)=\oplus_i\ZZ\varepsilon_i$.
Moreover $\sum\nu_i\varepsilon_i$ is dominant if and only if
$\nu_1\geq\cdots\geq\nu_n\geq 0$.
For $i\in [1;2n]$, set $\overline{i}=2n+1-i$.
The Weyl group $W$ of $G$ is a subgroup of the Weyl group $S_{2n}$ of
$\SL(V)$. More precisely
$$
W=\{w\in S_{2n}\,:\,w(\overline{i})=\overline{w(i)} \ \ \forall i\in [1;2n]\}.
$$
It is isomorphic to $S_n\ltimes(\ZZ/2\ZZ)^n$.
The group $Y(T)$ of 1-ps of $T$ identifies with $\ZZ^n$ by
$(a_1,\dots,a_n)\longmapsto(t\mapsto
\diag(t^{a_1},\dots,t^{a_n},t^{-a_n},\dots,t^{-a_1}))$.
The group $W$ acts on $Y(T)$ by permuting coordinates and changing the
signs of the coordinates. The dominant 1-ps are those satisfying 
$a_1\geq\cdots \geq a_n\geq 0$.

\subsection{Step~1: weights of $T$ in $\hlg/\lg$ and 1-ps}

The quotient $\hlg/\lg$ is isomorphic to $\bigwedge^2V^*/\CC\omega$ as a
$G=\Sp(V)$-module. Then, the set of weights of $T$ in $\hlg/\lg$ is
$$
 {\rm Wt}_T(\hlg/\lg)=\pm\{\varepsilon_i\pm\varepsilon_j\,:\,1\leq i<j\leq n\}.
$$

\begin{lem}
  \label{lem:adopsSpSl}
Let $n\geq 2$.
The dominant indivisible admissible 1-ps of $T$ for the pair
$(\Sp_{2n},\SL_{2n})$ are the following $n-1$ points of
$\ZZ^n$:
$$
\lambda_1=(1,0,\dots,0),\,\lambda_2=(1,1,0,\dots,0),\cdots,\lambda_{n-2}=(1,\dots,1,0,0)\mbox{ and
}\lambda_n=(1,\dots,1).
$$
\end{lem}

\begin{proof}
We easily check that each $\lambda_i$ in the statement is admissible. 
  Let $\lambda=(a_1,\dots,a_n)$ be a generic 1-ps. 
The equations $\langle\lambda,\alpha\rangle=0$ for some $\alpha\in
{\rm Wt}_T(\hlg/\lg)$ are $a_i=\pm a_j$ for some $i<j$.
We represent this  equation by a graph with two vertices indexed by $i$ 
and $j$ and one edge labelled by $\pm$. 
Consider a system of such equations which defines a line in
$\Chi(T)\otimes\QQ$.
We represent this system by a graph $\Gamma$ 
with vertices $i=1,\dots,n$
and edges labelled by $\pm$.

Each connected component of $\Gamma$ gives a subsystem in some variables $a_i$. By assumption, exactly one connected component 
$\Gamma_0$ gives a system with a line as solution and the other components 
have only the trivial solution.

Consider a connected subtree that contains any vertex of $\Gamma_0$. Up to $W$ we
may assume that the labels are $+$ for this subtree. 
The system associated to 
$\Gamma_0$ implies
$a_i=a_j$ for all vertices $i$ and $j$ of $\Gamma_0$. 
Since this system has solutions by assumption, it is spanned by the line $a_i=1$ for any $i$ in $\Gamma_0$.

The others connected components of the graph $\Gamma$ implies that $a_i=0$ if $i\not\in\Gamma_0$. 
Observe that these connected components encode at least two equations and have at least two vertices. 
The lemma is proved.
\end{proof}

\subsection{Step~2 : inversion sets}\label{sec:inclusionSpSl}


Let $r\in\{1,\dots, n-2\}$.
The inclusion of $G/P(\lambda_i)$ in $\hG/\hP(\lambda_i)$ is given by the following map 
$$
\begin{array}{cccc}
  \iota_r\,:&\Gr_\omega(r,2n)&\longto &\Fl(r,2n-r;2n)\\
&F&\longmapsto&(F,F^{\perp_{\omega_n}}).
\end{array}
$$
Set $F={\rm Span}(e_1,\dots,e_r)$, $G={\rm Span}(e_{r+1},\dots,e_{2n-r})$ and $\bar F={\rm Span}(e_{2n-r+1},\dots,e_{2n})$.
Then $V=F\oplus G\oplus\bar F$ is a $\hT$-stable decomposition, 
$\hL_r=S(\GL(F)\times \GL(G)\times \GL(\bar F))$ and the tangent 
space $T_{(F,F\oplus G)}\Fl(r,2n-r;2n)$ identifies with
$\hTa_r=\Hom(F,G)\oplus \Hom(F,\bar F) \oplus \Hom(G,\bar F)$.
Moreover $F^{\perp_{\omega_n}}=F\oplus G$, and $\omega_n$ identifies $\overline{F}$ 
with the dual of $F$.
The tangent space $T_{F}\Gr_\omega(r,2n)$ identifies with $\Ta_r={\rm Hom}(F,G)\oplus S^2F^*$.
The natural action of $L_r$ which is isomorphic to $\GL(F)\times\Sp(G)$ makes this
identification equivariant.

Using $\iota_r$, $\Ta_r$ identifies with the fixed point set in $\hTa_r$
of the involution: $(a,b,c)\longmapsto (^tc,^tb,^ta)$.
The weight spaces of  $\lambda_r$ in $\Ta_r$ and $\hTa_r$ are
$\Ta_r^{-1}={\rm Hom}(F,G)$, $\Ta_r^{-2}=S^2F^*$, 
$\hTa_r^{-1}=\Hom(F,G)\oplus \Hom(G,\bar F)$ and
$\hTa_r^{-2}= \Hom(F,\bar F) $.
In terms of matrices (with canonical bases), the inclusion of 
$\Ta_r^{-1}\subset\hTa_r^{-1}$  can be written as follows
$$
\begin{array}{ccc}
  {\rm Hom}(F,G)&\longto&\Hom(F,G)\oplus \Hom(G,\bar F)\\
A&\longmapsto&(A,J_r.{}^tA.\omega_{n-r}).
\end{array}
$$
 The inclusion of 
$\Ta_r^{-2}\subset\hTa_r^{-2}$  can be written as follows

\begin{eqnarray}
\label{iota2Ta}
\begin{array}{ccc}
  S^2F^*&\longto&\Hom(F,\bar F)\\
A&\longmapsto&J_r.A.
\end{array}
\end{eqnarray}

For $\lambda_n$ we find
$$
\begin{array}{cccc}
  \iota_n\,:&Gr_\omega(n,2n)&\longto &Gr(n,2n)\\
&F&\longmapsto&F.
\end{array}
$$
For $F={\rm Span}(e_1,\dots,e_n)$, 
$\Ta_n=\Ta_n^{-2}=S^2F^*$ is embedded in 
$\hTa_n=\hTa_n^{-2}=\Hom(F,\bar F)$ by formula \eqref{iota2Ta}. 

\bigskip
We draw $\Phi_r^{-1}$ and $\hat\Phi_r^{-1}$ as follows
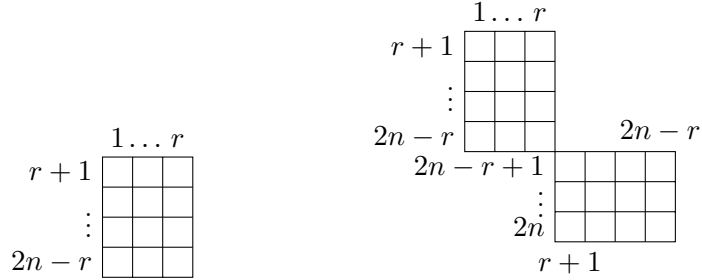
\begin{center}
 \begin{tikzpicture}[scale=0.4]
    \draw (0,1) rectangle (3,5) (1,1) -- (1,5) (2,1) -- (2,5);
\foreach \y in {2,...,4} {\draw (0,\y) -- (3,\y);};
\draw (0.5,5) node[above]{$1$} (1.5,5) node[above]{$\dots$}  (2.5,5) node[above]{$r$};
 \draw (0,4.5) node[left]{$r+1$} (0,3) node[left]{$\vdots$}  (0,1.5) node[left]{$2n-r$};
  \end{tikzpicture}
\hspace{2cm}
\begin{tikzpicture}[scale=0.4]
    \draw (-1,1) rectangle (2,5) (1,1) -- (1,5)  (0,1) -- (0,5);
\foreach \y in {2,...,4} {\draw (-1,\y) -- (2,\y);};

+\draw (2,1) rectangle (6,-2) (2,0) -- (6,0) (2,-1) -- (6,-1) ;
\foreach \x in {3,...,5} {\draw (\x,-2) -- (\x,1);};
\draw (-0.5,5) node[above]{$1$} (0.5,5) node[above]{$\dots$}  (1.5,5) node[above]{$r$};
\draw (2.5,-2) node[below]{$r+1$} 
(5.5,1) node[above]{$2n-r$};
\draw (-1,4.5) node[left]{$r+1$} (-1,3) node[left]{$\vdots$}  (-1,1.5) node[left]{$2n-r$};
\draw (2,0.5) node[left]{$2n-r+1$} (2,-0.5) node[left]{$\vdots$}  (2,-1.5) node[left]{$2n$};
  \end{tikzpicture}
\end{center}
where the box at line $i$ and column $j$ represents respectively the root
$\varepsilon_i-\varepsilon_j$ and 
$\hat\varepsilon_i-\hat\varepsilon_j$.
We draw $\Phi_r^{-2}$ and $\hat\Phi_r^{-2}$ as follows
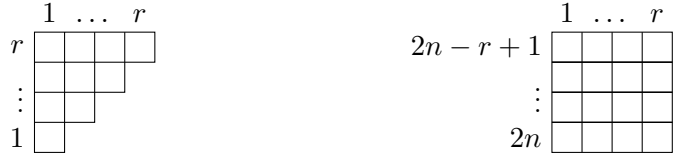
\begin{center}
\begin{tikzpicture}[scale=0.4]
    \draw (1,4) -- (1,0) -- (0,0) -- (0,4) -- (4,4) -- (4,3) --(0,3); 
\draw (3,4) -- (3,2) -- (0,2);
 \draw (2,4) -- (2,1) -- (0,1);
  \draw (0.5,4) node[above]{$1$} (2,4) node[above]{$\dots$}  (3.5,4) node[above]{$r$};
 \draw (0,3.5) node[left]{$r$} (0,2) node[left]{$\vdots$}  (0,0.5) node[left]{$1$};
  \end{tikzpicture}
\hspace{3cm}
\begin{tikzpicture}[scale=0.4]
    \draw (0,0) rectangle (4,4);
\foreach \x in {1,...,3} {\draw (\x,0) -- (\x,4);
\foreach \y in {1,...,3} {\draw (0,\y) -- (4,\y);};
};
 \draw (0.5,4) node[above]{$1$} (2,4) node[above]{$\dots$}  (3.5,4) node[above]{$r$};
 \draw (0,3.5) node[left]{$2n-r+1$} (0,2) node[left]{$\vdots$}  (0,0.5) node[left]{$2n$};
 \end{tikzpicture}  
\end{center}
where the box at line $i$ and column $j$ represents respectively the root
$-\varepsilon_i-\varepsilon_j$ and 
$\hat\varepsilon_i-\hat\varepsilon_j$.

The Schubert classes of $\Gr_\omega(r,2n)$ correspond bijectively with
the subsets $I$ of $\{1,\dots,2n\}$ with $r$ elements such
that $j\in I\ \Rightarrow\ 2n+1-j\not\in I$.
For such a class $I$ set 
$$
I^+=I\cap [1,n]\quad\mbox{and}\quad
I^-=\{2n+1-j\,|\,j\in I\cap[n+1,2n]\}.
$$
 The associated inversion set $\Phi(I)^{-1}$ is a Young diagram. 
 The set $\Phi(I)^{-2}$ is the upper part of a symmetric Young diagram.
 See Figure~\ref{fig:explePhiw}.
 
 \begin{figure}
 \begin{tikzpicture}[scale=0.3]
 \draw [fill=gray!60]  (0,0) -- (0,3) -- (3,3) -- (3,1) -- (4,1) -- (4,0) -- cycle;
    \draw (0,0) rectangle (4,4);
\foreach \x in {1,...,3} {\draw (\x,0) -- (\x,4);
\foreach \y in {1,...,3} {\draw (0,\y) -- (4,\y);};
};

 \end{tikzpicture}  
 \hspace{1cm}
 \begin{tikzpicture}[scale=0.3]
\draw [fill=gray!60] (0,0) -- (0,4) -- (1,4) -- (1,3) -- 
(3,3) -- (3,2) -- (2,2) --(2,1) -- (1,1) -- (1,0) -- cycle;
    \draw (1,4) -- (1,0) -- (0,0) -- (0,4) -- (4,4) -- (4,3) --(0,3); 
\draw (3,4) -- (3,2) -- (0,2);
 \draw (2,4) -- (2,1) -- (0,1);
  \end{tikzpicture}
  \hspace{1cm}
 
 \caption{$\Phi_r(I)^{-1}$ and $\Phi_r(I)^{-2}$ for $r=4$, $n=6$
 and $I=\{2,4,5,10\}$}
 \label{fig:explePhiw}
 \end{figure}
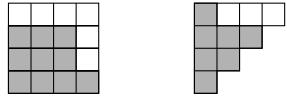
 
The Schubert classes of $\Fl(r,2n-r;2n)$ correspond bijectively with
the  pairs of subsets $J\subset K$ of $\{1,\dots,2n\}$ with $r$
and $2n-r$ elements. 
The associated inversion sets $\hat\Phi_r^{-1}$ and 
$\hat\Phi_r^{-2}$ are pairs of Young diagrams and Young diagrams respectively. See Figure~\ref{fig:explePhiw2}.

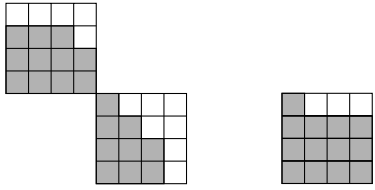
\begin{figure}
\begin{tikzpicture}[scale=0.3]
\draw [fill=gray!60] (-2,1) -- (-2,4) -- (1,4) -- (1,3) -- 
(2,3) -- (2,1)  -- cycle;
\draw [fill=gray!60] (2,-3) -- (2,1) -- (3,1) -- (3,0) -- 
(4,0) -- (4,-1) -- (5,-1) --(5,-3)  -- cycle;

    \draw (-2,1) rectangle (2,5) (1,1) -- (1,5)  (0,1) -- (0,5) (-1,1) -- (-1,5);
\foreach \y in {2,...,4} {\draw (-2,\y) -- (2,\y);};

\draw (2,1) rectangle (6,-3) (2,0) -- (6,0) (2,-1) -- (6,-1) (2,-2) -- (6,-2);
\foreach \x in {3,...,5} {\draw (\x,-3) -- (\x,1);};
  \end{tikzpicture}
  \hspace{1cm}
  \begin{tikzpicture}[scale=0.3]
  \draw [fill=gray!60] (0,0) -- (0,4) -- (1,4) -- (1,3) -- 
(4,3) -- (4,0)  -- cycle;
    \draw (0,0) rectangle (4,4);
\foreach \x in {1,...,3} {\draw (\x,0) -- (\x,4);
\foreach \y in {1,...,3} {\draw (0,\y) -- (4,\y);};
};
 \end{tikzpicture}

\caption{$\Phi_r(I\subset J)^{-1}$ and $\Phi_r(I\subset J)^{-2}$ for $r=4$, $n=6$, $I=\{2,4,5,7\}$ and $J=I\cup\{1,6,9,12\}$}
\label{fig:explePhiw2}
\end{figure} 

\subsection{Step~3: inequalities}

The weights of $G$ and $\hG$ are expressed using the standard bases. In particular, a pair $(\nu,\hnu)$ of dominant weights is given by $3*n-1$
integers $(\nu_i)_{1\leq i\leq n}$ and $(\hnu_i)_{1\leq i\leq 2*n-1}$ satisfying  $\nu_1\geq\cdots\geq \nu_n\geq 0$ and 
$\hnu_1\geq\cdots\geq \hnu_{2n-1}\geq 0$. 
The inequality corresponding to the pair $(I,J\subset K)$ of Schubert classes such that 
$$
\sigma_I.\iota_i^*(\sigma_{J\subset K})\neq 0
$$
is
\begin{eqnarray}
  \label{eq:3}
  \sum_{i\in I^-}\nu_i +\sum_{j\in J}\hat\nu_j
  \leq  \sum_{j\in I^+}\nu_j+\sum_{j\not\in K}\hat\nu_j,
\end{eqnarray}
where by convention $\hnu_{2n}=0$.

For example, the Schubert classes 
$[\hG/\hP]$ and $[G/P]$ correspond to $I=J=\{2n-r+1,\dots,2n\}$ and
$K=\{r+1,\dots,2n\}$. 
The associated inequality is 
 
\begin{eqnarray}
  \label{eq:4}
  \sum_{i=1}^r \nu_i +\sum_{j=2n-r+1}^{2n-1}\hat\nu_j\leq 
  \sum_{k=1}^r \hnu_k,
\end{eqnarray}
for any $r=1,\dots,n-2$ or $r=n$.
The case $r=n-1$ gives redundant inequalities.

\subsection{Step~4: Levi-movability}
In Section~\ref{sec:inclusionSpSl}, we explain, for any $r$, how to realize $\Ta_r$ as a subspace of $\hTa_r$, the action of $\hL_r$ on $\hTa_r$,
and how to encode the inversion sets. This is used in our Sage program 
to determine the L-movable pairs $(I,(J\subset K))$.

\subsection{The lattice $\ZZ LR(\Sp_{2n},\SL_{2n})$}

The center of $Sp_{2n}$ is $\{\pm I_{2n}\}$. Then $(\nu,\hnu)$ belongs to $\ZZ LR(\Sp_{2n},\SL_{2n})$ 
if and only if 
$\sum_{i=1}^n\nu_i+\sum_{j=1}^{2n-1}\hnu_j$ is even.

\subsection{Some extremal rays}
Recall that the fundamental weight are $\varpi_i=\varepsilon_1+\cdots+\varepsilon_i$.

\begin{propo}\label{propo:expleRaysSpSl}
Let $V$ be a $2n$-dimensional vector space endowed with a symplectic
form.
Convention: $V_{\varpi_0}$ denotes the trivial representation.
The following inclusions and their dual give extremal rays (and belong to the Hilbert
basis) of $\QQ_{\geq 0}LR(\Sp_{2n},\SL_{2n})$:
\begin{enumerate}
\item $\CC\subset V(\hat\varpi_{2k})$ with $k=1,\dots,n$;
\item $V({\varpi_i})\subset V(\hat\varpi_j)$ with $j\geq i$ and $j-i$
  even; 
\item $V({\varpi_2})\subset V(\hat\varpi_1+\hat\varpi_{2n-1})$.
\end{enumerate}

The first two items give the only rays with $\hnu$ fundamental.
\end{propo}

\begin{proof}
  The first one is a ray of the dominant chamber. 
The second one is the only half-line in
$\QQ\varpi_i\oplus\QQ\hat\varpi_j$.
The last one is the only half-line in
$\QQ\varpi_2\oplus\QQ\hat\varpi_1\oplus \QQ\hat\varpi_{2n-1}$.
\end{proof}

\subsection{The smallest case: $\Sp_4$ in $\SL_4$}

  The group $\Sp_4$ is $\Spin_5$ and $\SL_4$ is 
$\Spin_6$. In particular, the semigroup is recalled in the example of 
Section~\ref{sec:LRgroup}.

\begin{propo}\label{th:sp4sl4}
The minimal list of inequalities for $\QQ_\geq LR(\Sp_4,\SL_4)$ is 
\begin{enumerate}
\item $\hnu_1-\hnu_2+\hnu_3\leq \nu_1+\nu_2\leq\hnu_1+\hnu_2-\hnu_3$;
\item $\max(-\hnu_1+\hnu_2+\hnu_3,\hnu_1-\hnu_2-\hnu_3)\leq \nu_1-\nu_2\leq\hnu_1-\hnu_2+\hnu_3$;
\end{enumerate}
The indivisible generators of the 5 extremal rays form the Hilbert basis of $\QQ_\geq LR(\Sp_4,\SL_4)\cap \ZZ LR(\Sp_4,\SL_4)$.
The 5 corresponding inclusions $V_G(\nu)\subset V_\hG(\hnu)$ are particular cases of Proposition~\ref{propo:expleRaysSpSl}.
\end{propo}


\subsection{The case $\Sp_6$ in $\SL_6$}

\begin{propo}\label{th:sp6sl6}
The cone $\QQ_\geq LR(\Sp_6,\SL_6)$ is the set of 
$(\nu,\hnu)$ such that 
\begin{enumerate}
\item $\max(\hnu_1-\hnu_2,\hnu_3-\hnu_4,\hnu_5)\leq\nu_1\leq \hnu_1$;
\item $\nu_2\leq \min(\hnu_1-\hnu_5,\hnu_2)$;
\item $\nu_3\leq \min(\hnu_1-\hnu_4,\hnu_2-\hnu_5,\hnu_3)$;
\item $\hnu_1-\hnu_2+\hnu_3-\hnu_4+\hnu_5\leq \nu_1+\nu_2+\nu_3\leq \hnu_1+\hnu_2+\hnu_3-\hnu_4-\hnu_5$;
\item $\max(-\hnu_1-\hnu_2+\hnu_3+\hnu_4+\hnu_5,
\hnu_1-\hnu_2-\hnu_3-\hnu_4+\hnu_5,
-\hnu_1+\hnu_2-\hnu_3+\hnu_4-\hnu_5
)\leq \nu_1-\nu_2-\nu_3\leq \hnu_1-\hnu_2+\hnu_3-\hnu_4+\hnu_5$;
\item$\max(-\hnu_1+\hnu_2-\hnu_3+\hnu_4+\hnu_5,
\hnu_1-\hnu_2-\hnu_3+\hnu_4-\hnu_5,
-\hnu_1+\hnu_2+\hnu_3-\hnu_4-\hnu_5)\leq \nu_1-\nu_2+\nu_3\leq \min(\hnu_1-\hnu_2+\hnu_3+\hnu_4-\hnu_5,
\hnu_1+\hnu_2-\hnu_3-\hnu_4+\hnu_5)$;
\item $\max(\hnu_1-\hnu_2-\hnu_3+\hnu_4+\hnu_5,-\hnu_1+\hnu_2+\hnu_3-\hnu_4+\hnu_5,\hnu_1-\hnu_2+\hnu_3-\hnu_4-\hnu_5)\leq \nu_1+\nu_2-\nu_3\leq \hnu_1+\hnu_2-\hnu_3+\hnu_4-\hnu_5$;
\item $\nu_1\geq\nu_2\geq\nu_3\geq 0$  (dominance of $\nu$);
\item $\hnu_1\geq\hnu_2\geq\hnu_3\geq\hnu_4\geq\hnu_5\geq 0$  (dominance of $\hnu$).
\end{enumerate}
Moreover this list of inequalities is not redundant.
The 15 extremal rays of the cone are respectively generated by the following vectors written in row.
$$
\begin{array}{c@{\ }c@{\ }c@{\ }c@{\ }c@{\ }c@{\ }c@{\ }c@{\quad}|@{\quad}c@{\ }c@{\ }c@{\ }c@{\ }c@{\ }c@{\ }c@{\ }c@{\quad}|@{\quad}c@{\ }c@{\ }c@{\ }c@{\ }c@{\ }c@{\ }c@{\ }c}
 0 & 0 & 0 & 1 & 1 & 1 & 1 & 0&
 1 & 1 & 1 & 3 & 2 & 2 & 1 & 1&
 0 & 1 & 0 & 2 & 1 & 1 & 0 & 0 \\

 0 & 0 & 0 & 1 & 1 & 0 & 0 & 0&
  1 & 1 & 0 & 1 & 1 & 1 & 1 & 0 &
   1 & 0 & 0 & 1 & 1 & 1 & 1 & 1 \\

 1 & 1 & 1 & 1 & 1 & 1 & 0 & 0&  
 1 & 1 & 0 & 2 & 2 & 2 & 1 & 1 & 
 1 & 0 & 0 & 1 & 1 & 1 & 0 & 0 \\

 1 & 1 & 1 & 2 & 2 & 1 & 1 & 1 &
  1 & 1 & 0 & 1 & 1 & 0 & 0 & 0 & 
  1 & 0 & 0 & 1 & 0 & 0 & 0 & 0 \\

 1 & 1 & 1 & 2 & 1 & 1 & 1 & 0 & 
 1 & 1 & 0 & 2 & 1 & 1 & 1 & 1 &
 2 & 1 & 1 & 2 & 2 & 1 & 1 & 0 
\end{array}
$$

These vectors form the Hilbert basis of the cone $\QQ_\geq LR(\Sp_6,\SL_6)$ in $\ZZ LR(\Sp_6,\SL_6)$.
They correspond to  inclusions of Proposition~\ref{propo:expleRaysSpSl} and the  following ones:
\begin{enumerate}
\item $V(\varpi_3)$ in $V(\hat\varpi_1+\hat\varpi_4)$ and its dual;
\item $V(\varpi_3)$  in $V(\hat\varpi_1+\hat\varpi_3+\hat\varpi_5)$;
\item $V(\varpi_2)$  in $V(\hat\varpi_1+\hat\varpi_3)$ and its dual;
\item $V(\varpi_1)$  in $V(\hat\varpi_3+\hat\varpi_5)$.
\end{enumerate}
\end{propo}

\subsection{The case $\Sp_8$ in $\SL_8$}

For $\lambda_1$, $\lambda_2$ and $\lambda_4$, we obtain respectively $14$, $47$ and $53$ L-movable pairs.
With the 11 inequalities of
dominance this gives 125 inequalities. The following one is the only one to be 
redundant 
\begin{eqnarray}
\label{eq:redundant8}
\nu_1-\nu_2+\nu_3-\nu_4\geq -\hnu_1+\hnu_2-\hnu_3+\hnu_4-\hnu_5+\hnu_6-\hnu_7
\end{eqnarray}
it is associated to the following well covering pair $(w,\hat w)$ of $LG(4,8)$ in
$Gr(4,8)$ such that

\begin{center}
$\Phi(w)=$
\begin{tikzpicture}[scale=0.3]
\draw [fill=gray!60] (0,0) -- (0,4) -- (1,4) -- (1,3) -- (2,3) --
(2,1) -- (1,1) -- (1,0) -- cycle;
    \draw (1,4) -- (1,0) -- (0,0) -- (0,4) -- (4,4) -- (4,3) --(0,3); 
\draw (3,4) -- (3,2) -- (0,2);
 \draw (2,4) -- (2,1) -- (0,1);
  \end{tikzpicture}
\hspace{1cm}and
\hspace{1cm}
$\Phi(\hw)=$
\begin{tikzpicture}[scale=0.3]
\draw [fill=gray!60] (0,0) -- (0,4) -- (1,4) -- (1,3) -- (2,3) --
(2,2) -- (3,2) -- (3,1) -- (4,1) -- (4,0) -- cycle;
    \draw (0,0) rectangle (4,4);
\foreach \x in {1,...,3} {\draw (\x,0) -- (\x,4);
\foreach \y in {1,...,3} {\draw (0,\y) -- (4,\y);};
};
 \end{tikzpicture}  
.
\end{center}

The rays of the face corresponding to the redundant inequality~\eqref{eq:redundant8} are the 8 following vectors.
$$
\begin{array}{rrrrrrrrrrr}
0&0&0&0&1&1&0&0&0&0&0 \\
 0&0&0&0&1&1&1&1&0&0&0 \\
 0&0&0&0&1&1&1&1&1&1&0 \\
 1&1&0&0&1&1&0&0&0&0&0 \\
 1&1&0&0&1&1&1&1&0&0&0 \\
 1&1&0&0&1&1&1&1&1&1&0 \\
 1&1&1&1&1&1&1&1&0&0&0 \\
 1&1&1&1&2&2&1&1&1&1&0 
\end{array}
$$

Using 4ti2, we check that the Hilbert basis consists of the  49 generators of rays.
The semigroup is saturated by the PRV Theorem (see \cite{MPR})
and by computer checking for the 4
following cases (remark that the last one is the dual of the 3rd one ).
$$        	
\begin{array}{rrrrrrrrrrr}
 1 & 1 &
1 & 1 & 2 & 2 & 2 & 1 & 1 & 0 & 0 \\
2 & 2 &
2 & 0 & 3 & 3 & 2 & 2 & 1 & 1 & 0 \\
 2 & 2 &
1 & 1  & 3 & 3 & 2 & 2 & 2 & 0 & 0 \\
2 & 2 &
1 & 1 & 3 & 3 & 3 & 1 & 1 & 1 & 0 
\end{array}
$$

\subsection{The case $\Sp_{10}$ in $\SL_{10}$}

We have 534 $L$-movable pairs. With the 14 dominancy inequalities, they give 548 inequalities, including 29 redundant ones. We obtain 194 rays.
The Hilbert basis consist of the set of primitive  generators of the rays.
Note that 4ti2 needed about 250 hours to make this computation.
The PRV Theorem  \cite{MPR} shows that 141 elements of this Hilbert basis belong to 
$LR(G,\hG)$.
Using the fact that $V\subset\hat V$ if and only if $V\subset\hat V^*$, the list of
remaining cases can be reduced to  31 cases. Using Sage, we check that these 31 points belong to  $LR(G,\hG)$. 
Some details are available in authors' web pages.

\subsection{Final remarks}

These examples raise a natural question (in addition to the question of Section~\ref{sec:intro}). Indeed, we remark that the Hilbert basis equals the set of primitive generators of rays for $n=2,\,3,\,4$ and $5$. Is this fact
  true for any $n$?\\
  
  For all the computed examples,  the cones $\QQ_{\geq 0}LR(G,\hG)$ have few rays compared to the number of facets. For example, $\QQ_{\geq 0}LR(Sp_8,Sl_8)$ has 49 rays 
  and 124 facets, and  $\QQ_{\geq 0}LR(Sp_{10},Sl_{10})$ has 194 rays   and 531 facets.
  This suggests that it could be interesting to study these rays from a theoretic point of view, whereas the litterature concentrates on the facets ?\\
  
In the programs used to compute the inequalities, the rays, the Hilbert basis and to check the saturation property, the most expensive in time is the computation of the Hilbert basis with 4ti2. That is why, we do not try to study the cases for $n\geq 6$. 
Another limiting factor is the computation of the inversion sets. But, here our programs are really not optimal. If someone is interested in computing the inequalities for $n\geq 6$, he could considerably improve them to do it in a more reasonable time.

\bibliographystyle{amsalpha}
\bibliography{saturationGhG}

\end{document}